\documentclass[11pt]{amsart}
\usepackage[all]{xy}

\usepackage{amsmath}
\usepackage{amsfonts}
\usepackage{amssymb}

\DeclareFontEncoding{OT2}{}{} 
\newcommand{\textcyr}[1]{%
 {\fontencoding{OT2}\fontfamily{wncyr}\fontseries{m}\fontshape{n}
 \selectfont #1}}
\newcommand{\Sha}{{\mbox{\textcyr{Sh}}}}

\newcommand{\Tha}{{\mbox{\textcyr{Q}}}}

\numberwithin{equation}{section}

\def\br{\text{Br}}
\def\bz{{\mathbb Z}\,}
\def\bq{{\mathbb Q}}
\def\bg{{\mathbb G}}
\def\spec{{\rm{Spec}}\,}
\def\fb{\overline{F}}
\def\ut{\widetilde{U}}
\def\img{{\rm{Im}}\,}
\def\tto{\tt^{\le\circ}}
\def\sqo{{\s Q}^{\le\circ}}
\def\spo{{\s P}^{\le\circ}}
\def\ba{\mathbb A}

\def\be{\kern -.1em}
\def\lbe{\kern -.05em}

\def\s{\mathcal }
\def\ra{\rightarrow}
\def\e{\kern 0.08em}
\def\le{\kern 0.03em}
\def\ng{\kern -0.04em}

\def\tt{\widetilde{\s T}}

\def\fc{\fb\!\phantom{.}^{\lbe *}}
\def\dimn{\text{dim}\,}
\def\sto{{\s T}^{\le\circ}}
\def\krn{{\rm{Ker}}\,}
\def\cok{{\rm{Coker}}\,}

\newtheorem{lemma}{Lemma}[section]
\newtheorem{theorem}[lemma]{Theorem}
\newtheorem{corollary}[lemma]{Corollary}
\newtheorem{proposition}[lemma]{Proposition}
\theoremstyle{definition}
\newtheorem{definition}[lemma]{Definition}
\theoremstyle{remark}
\newtheorem{remark}[lemma]{Remark}

\newtheorem{example}[lemma]{Example}

\begin{document}

\title[Class groups of tori and capitulation]{On N\'eron-Raynaud class groups of tori and the Capitulation Problem}


\author{Cristian D. Gonz\'alez-Avil\'es}
\address{Departamento de Matem\'aticas, Universidad de La Serena, Chile}
\email{cgonzalez@userena.cl}


\thanks{The author is partially supported by Fondecyt grant
1080025}

\maketitle

\begin{abstract} We discuss the Capitulation Problem for N\'eron-Raynaud
class groups of tori over global fields $F$ and obtain
generalizations of the main results of \cite{CAmb}. We also show
that short exact sequences of $F$-tori induce long exact sequences
involving the corresponding N\'eron-Raynaud class groups. For
example, we show that the N\'eron-Raynaud class group of any
$F$-torus which is split by a metacyclic extension of $F$ can be
``resolved'' in terms of classical ideal class groups of global
fields.
\end{abstract}

\section{Introduction}

Given an algebraic torus $T$ over a number $F$ and an integral model
$\s H$ of $T$, one can define a class group $C(\s H)$ of $\s H$ by
extending in a natural way the well-known adelic definition of the
classical ideal class group of $F$. Several authors
\cite{K,O,PR,Sh2,V} have studied the class group of a particular
model of $T$, namely the so-called {\it standard model} of $T$, and
obtained interesting results for norm tori and their duals. The
corresponding proofs are rather involved due to the unwieldy nature
of the adelic definition mentioned above. Fortunately Ye.Nisnevich,
in his thesis \cite{Nis}, obtained a cohomological intepretation of
$C(\s H)$ which produces significant simplifications in the study of
these groups, as already demonstrated by M.Morishita in \cite{M} (we
should mention that this author studied the class group of a certain
model of the norm torus which differs from both the standard model
mentioned above and the connected N\'eron-Raynaud model considered
below). For the convenience of the reader, we have included in
Section 3 of this paper an overview of Nisnevich's construction. In
the case of the {\it connected N\'eron-Raynaud} model $\s H=\sto$ of
an $F$-torus $T$ (over any global field $F$), Nisnevich's
cohomological interpretation quickly leads to a {\it non-adelic}
description of the N\'eron-Raynaud class group $C(\sto)$ (see
\cite{mrl}, \S3, or Proposition 4.1 below) which generalizes the
well-known {\it ideal-theoretic} definition of the ideal class group
of a global field. This alternative description is superbly
well-suited for studying N\'eron-Raynaud class groups of {\it
arbitrary tori}. In this paper we discuss two problems related to
these groups. To explain them, we introduce the following notations.

Let $F$ be a global field and let $S$ be any nonempty finite set of
primes of $F$ containing the archimedean primes in the number field
case. Let $\sto$ denote the (fiberwise) identity component of the
N\'eron-Raynaud model of $T$ over $U=\spec O_{F,S}$ and let $C_{T,\e
F,\e S}=C(\sto)$ be the corresponding class group. The first problem
that we consider is a natural extension of the classical {\it
$S$-Capitulation Problem} for ideal class groups of number fields.
Namely, given a finite Galois extension $K/F$ of $F$ with Galois
group $G$, describe the kernel and cokernel of the
induced $S$-capitulation map $j_{\e T,\e K/F,\e S}\colon C_{\e T,\e
F,\e S}\ra C_{\e T,\e K,\e S_{K}}^{\e G}$, where $S_{K}$ denotes the
set of primes of $K$ lying above the primes of $S$. Let $\ut=\spec\s
O_{K,\e S_{K}}$, write $\tto$ for the identity component of the
N\'eron-Raynaud model of $T_{K}$ over $\ut$ and let
$$
H^{1}\ng\big(G,\tto\be(\ut)\big)^{\e\prime}=\krn\!\be\left[H^{1}\ng\big(G,\tto\be(\ut)\big)\ra H^{1}(G,T(K))\right],
$$
where the map involved is induced by the inclusion $\tto\be(\ut)\ra
T(K)$. Further,
for each prime $v\notin S$, let
$H^{1}(G_{w_{v}},\tto(\s O_{w_{v}}))^{\e\prime}$ be the analogous
group associated to $T_{K_{w_{v}}}$, where $w_{v}$ is a
previously-selected prime of $K$ lying above $v$, $\s O_{w_{v}}$ is
the ring of integers of the completion $K_{w_{v}}\!$ of $K$ at
$w_{v}$ and $G_{w_{v}}=\text{Gal}(K_{w_{v}}/F_{v})$. There exists a
canonical localization map
$$
\lambda_{S}\colon H^{1}\ng\big(G,\tto\be(\ut)\big)^{\e\prime}\ra
\displaystyle\bigoplus_{v\notin S} H^{1}(G_{w_{v}},\tto(\s
O_{w_{v}}))^{\e\prime}.
$$
Our first result is the following generalization of \cite{CAmb},
Theorem 2.4.

\begin{theorem} Assume that $S$ contains all primes of $F$ where $T$ has bad reduction. Then there exists a canonical exact sequence
$$\begin{array}{rcl}
0\ra\krn j_{\e T,\e K/F,\e S}\ra H^{1}\ng\big(G,\tto\lbe(\ut)\big)^{\e\prime}&\overset{\lambda_{S}}\longrightarrow&
\displaystyle\bigoplus_{v\notin S} H^{1}(G_{w_{v}},\tto(\s
O_{w_{v}}))^{\e\prime}\\\\
&\longrightarrow &\cok j_{\e T,\e K/F,\e S}^{\e\prime}\ra 0,
\end{array}
$$
where $j_{\e T,K/F,S}^{\e\prime}$ is a variant of $j_{\e T,K/F,S}$
defined in Section 4.
\end{theorem}

Regarding the cokernel of $j_{\e T,K/F,S}$, we obtain the following generalization of \cite{CAmb}, Theorem 3.3.

\begin{theorem} Assume that $T$ splits over $K$ and let $R$ be the set of primes of $F$ which ramify in $K$. Then there exists a canonical exact sequence
$$
\begin{array}{rcl}
0&\ra&\cok\widetilde{\jmath}_{\e T,\e K/F,\e S}^{\,\,\prime}\ra
\cok j_{\e T,\e K/F,\e S}\ra H^{2}\ng\big(G,\tto\be(\ut)\big)\ra B_{S}(G,T)\\\\
&\ra& H^{1}\big(G, C_{\e T,\e K,\e S_{K}}\big)\ra H^{3}\ng\big(G,\tto\be(\ut)\big),
\end{array}
$$
where $\widetilde{\jmath}_{T,\e K/F,\e S}^{\,\,\prime}$ is a variant
of $j_{\e T,\e K/F,\e S}$ defined in Section 5 and the group
$B_{S}(G,T)$ fits into an exact sequence
$$
0\ra\Sha^{2}(T)\ra B_{S}(G,T)\ra\displaystyle\bigoplus_{v\e\in\e
S\e\cup\le R}H^{2}(G_{w_{v}},T(K_{w_{v}}))^{\e\prime}
\ra\widehat{H}^{\e 0}(G,X)^{D}.
$$
Here $\!\Sha^{2}(T)$ and $X$ are, respectively, the second
Tate-Shafarevich group of $T$ and the $G$-module of characters of
$T$, $\widehat{H}^{\e 0}(G,X)^{D}$ is the Pontryagin dual of the
0-th Tate cohomology group of $X$ and the groups
$H^{2}(G_{w_{v}},T(K_{w_{v}}))^{\e\prime}$ are certain variants of
$H^{2}(G_{w_{v}},T(K_{w_{v}}))$ defined in Section 5.
\end{theorem}

The second question addressed in this paper is the following one:
given a short exact sequence of $F$-tori
\begin{equation}\label{uno}
0\ra T_{1}\ra T_{2}\ra T_{3}\ra 0,
\end{equation}
how are the N\'eron-Raynaud $S$-class groups $C_{T_{i},\e F,\e S}$
related? We can answer this question for certain types of sequences
\eqref{uno}\footnote{\e Our methods should lead to an answer to the above question (for any sequence \eqref{uno}) provided $S$ contains
all primes of $F$ which are wildly ramified in the minimal splitting field of $T_{1}$.}. In order to explain our results, we first recall that
an $F$-torus is called {\it quasi-trivial} if its module of
characters is a permutation $G_{\lbe F}$-module, where $G_{\lbe F}$
is the absolute Galois group of $F$. It is called {\it invertible}
if it is a direct factor of a quasi-trivial torus. Let $T$ be an
$F$-torus and assume that $T$ admits an {\it invertible resolution},
i.e., there exists an exact sequence
\begin{equation}\label{dos}
0\ra T_{1}\ra Q\ra T\ra 0,
\end{equation}
where $Q$ is quasi-trivial and $T_{1}$ is invertible. This is the
case, for example, if $T$ is split by a {\it metacyclic} extension
of $F$. Then the following holds.

\begin{theorem} The exact sequence \eqref{dos} induces an exact sequence of finitely
generated abelian groups
$$
0\ra\sto_{1}(U)\ra\sqo(U)\ra\sto(U)\ra C_{\e T_{1},\le F,S}\ra C_{\e
Q,\e F,S}\ra C_{\e T,\e F,S}\ra 0,
$$
where $\s T_{1}, \s Q$ and $\s T$ denote, respectively, the N\'eron-Raynaud models of $T_{1}, Q$ and $T$ over $U$.
\end{theorem}

See Section 6 for applications of the above result to duals of norm
tori.

\smallskip

Certainly, the class of $F$-tori which admit an invertible
resolution is a rather restricted one. But {\it any} $F$-torus $T$
admits a {\it flasque resolution}
\begin{equation}\label{tres}
0\ra T_{1}\ra Q\ra T\ra 0,\bibliographystyle{mrl}
\end{equation}
where $Q$ is quasi-trivial and $T_{1}$ is flasque. We recall that an
$F$-torus $T_{1}$ with module of characters $X_{1}$ is called {\it
flasque} if $\widehat{H}^{-1}(H,X_{1})=0$ for every open subgroup
$H$ of $G_{\lbe F}$. Then there exists a naturally-defined variant
$C_{\e T,\e F,\e S}^{R}$ of $C_{\e T,\e F,\e S}$, which we call the
{\it $R$-equivalence class group of $T$}, so that the following result
holds.

\begin{theorem} Assume that $S$ contains all primes of $F$ where $T_{1}$ in
\eqref{tres} has bad reduction. Then \eqref{tres} induces an exact sequence
$$
0\ra\sto_{1}(U)\ra\sqo(U)\ra R\,\sto(U)\ra C_{\e T_{1},\e F,\e S}\ra
C_{\e Q,\e F,\e S}\ra C_{\e T,\e F,\e S}^{R}\ra 0,
$$
where $R\,\sto(U)$ denotes the subgroup of $\sto(U)$ of all elements which are $R$-equivalent to 1.
\end{theorem}

In fact, we obtain a more general result. See Theorem 7.4.

This paper discusses other issues that seem interesting,
e.g., the computation of {\it Ono invariants} for certain types of tori (see Sections 6 and 8).

\smallskip

This paper is divided into 8 Sections. Section 2 is preliminary. In
Section 3 we present Ye.Nisnevich's cohomological interpretation of
the class set of an affine group scheme. We hope that this Section will be useful to other
researchers in this area. Sections 4 and 5 contain
the proofs of Theorems 1.1 and 1.2. Sections 6 and 7, which to a large extent can be
read independently of Sections 4 and 5, contain
the proofs of Theorems 1.3 and 1.4. Finally, Section 8 discusses
class groups of norm tori.

\smallskip

The problems discussed in this paper admit interesting analogues for
abelian varieties. We hope to discuss them in future publications.
For an indication of the type of problems to be considered, see
\cite{leiden-talk}.

\section*{Acknowledgements}
I thank B.Edixhoven for his help with the proof of Lemma 2.1. I also
thank him, P.Gille, D.Lorenzini and X.Xarles for helpful
comments. Most of this paper was written during visits to Leiden
University and University of Bordeaux in the winter of 2009. I
thank both institutions for their hospitality.

\section{Preliminaries}

Let $F$ be a global field, i.e. $F$ is a finite extension of
${\mathbb Q}$ or is finitely generated and of transcendence degree 1
over a finite field of constants. Let $S$ be any nonempty finite set
of primes of $F$ which contains the archimedean primes in the number
field case. Let $K/F$ be a finite Galois extension of $F$ with
Galois group $G$. We fix a separable algebraic closure $\fb$ of $F$
containing $K$ and write $G_{\be K}$ for
$\text{Gal}\big(\e\fb/K\big)$. Further, we will write $S_{\be K}$
for the set of primes of $K$ lying above the primes in $S$. The ring
of $S$ (resp., $S_{\be K}$)-integers of $F$ (resp., $K$) will be
denoted by $\s O_{F,\e S}$ (resp., $\s O_{K,\e S_{K}}$). Sometimes
it is convenient to assume that $K/F$ is only separable, in which
case we keep these notations. Let $U=\spec\s O_{F,\e S}$ and
$\ut=\spec{\s O}_{K,\e S_{K}}$. For every $v\notin S$, the
henselization (resp., completion) of the local ring of $U$ at $v$
will be denoted by ${\s O}^{\e h}_{v}$ (resp., ${\s O}_{v}$), and
$F_{v}^{\e h}$ and $F_{v}$ will denote the corresponding fields of
fractions. We will write $k(v)$ for the residue field of $U$ at $v$
and $i_{v}\colon\spec k(v)\ra U$ for the corresponding closed
immersion. For any prime $v$ of $F$, we fix once and for all a prime
$w_{v}$ of $K$ lying above $v$ and write
$G_{w_{v}}=\text{Gal}(K_{w_{v}}/F_{v})$ for the decomposition group
of $w_{v}$ in $K/F$. The inertia subgroup of $G_{w_{v}}$ will be
denoted by $I_{w_{v}}$. Let $\overline{w}_{v}$ be a fixed prime of
$\fb$ lying above $w_{v}$. Then the completion of $\fb$ at
$\overline{w}_{v}$, $\fb_{\be\overline{w}_{v}}$, is a separable
algebraic closure of $F_{v}$ containing $K_{w_{v}}$. We set
$I_{\e\overline{w}_{v}}=\text{Gal}\big(\e\fb_{\be\overline{w}_{v}}/
K_{w_{v}}^{\e\text{nr}}\e\big)$ and
$I_{\e\overline{v}}=\text{Gal}\big(\e\fb_{\be\overline{w}_{v}}/
F_{v}^{\e\text{nr}}\e\big)$, where $K_{w_{v}}^{\e\text{nr}}$ (resp.,
$F_{v}^{\e\text{nr}}$) is the maximal unramified extension of
$K_{w_{v}}$ (resp., $F_{v}$) inside $\fb_{\be\overline{w}_{v}}$.
Clearly, $I_{\e\overline{w}_{v}}$ is a subgroup of
$I_{\e\overline{v}}$ and there exist canonical isomorphisms
$I_{\e\overline{v}}/I_{\e\overline{w}_{v}}=\text{Gal}(K_{w_{v}}^
{\e\text{nr}}/F_{v}^{\e\text{nr}})=I_{w_{v}}$. We will write
$G(w_{v})$ for $\text{Gal}(k(w_{v})/k(v))$, which will be identified
with $G_{w_{v}}/I_{w_{v}}$. Further, we will write $e_{v}$ for the
ramification index of $v$ in $K$, i.e., $e_{v}=[K_{w_{v}}^
{\e\text{nr}}\colon F_{v}^{\e\text{nr}}]$. If $K/F$ is a finite
separable extension, $v$ is a non-archimedean prime of $F$ and $w$
is a prime of $K$ lying above $v$, then we will write $e_{w}$ for
$[K_{w}^ {\e\text{nr}}\colon F_{v}^{\e\text{nr}}]$.

Let $T$ be an $F$-torus and let $\tt$ be the N\'eron-Raynaud model
of $T_{K}$ over $\ut$. Then $\tt$ is a smooth and separated
$\ut$-group scheme which is locally of finite type and represents
the sheaf $\widetilde{\jmath}_{\e *}T_{K}$ on the small smooth site
over $\ut$. See \cite{BLR}, Proposition 10.1.6, p.292. Now, if $\s
T$ denotes the N\'eron-Raynaud model of $T$ over $U$, there exists a
canonical base-change map $\s T\times_{U}\ut\ra\tt$ which is an
isomorphism if $\ut/U$ is \'etale \cite{BLR}, \S 7.2, Theorem 1(i),
p.176. Let $\tto$ (resp., $\s T^{\le\circ}$) denote the (fiberwise)
identity component of $\tt$ (resp., $\s T$). Then $\tto$ is an
affine smooth $U$-group scheme of finite type (see \cite{KM},
Proposition 3, p.18, and \cite{BLR}, p.290, line 6). For each prime
$w\notin S_{K}$, let $\Phi_{w}=i_{w}^{*}\be\big(\le\tt/\tto\le\big)$
be the \'etale $k(w)$-sheaf of connected components of $\tt$ at $w$.
Then $\Phi_{w}\be\big(\e\overline{k(w)}\e\big)$ is a finitely
generated $G_{\lbe k(w)}$-module \cite{X}, Proposition 2.18. For
each prime $v\notin S$, we will write $\Phi_{v}$ for the \'etale
sheaf of connected components of $\s T$ at $v$. When necessary to
avoid confusion, we will write $\Phi_{v}(T)$ for $\Phi_{v}$ and
$\Phi_{w}(T_{K})$ for $\Phi_{w}$. The {\it N\'eron-Raynaud $S$-class
group of $T$} is the finite group $C_{\e T,\e F,\e S}=C\big(\e\s
T^{\e\circ}\lbe\big)$ (see \cite{C}, \S 1.3, for the finiteness
assertion). Its cardinality will be denoted by $h_{\e T,F,S}$.

We recall from \cite{mrl} the definition of the capitulation map on
$C_{\e T,\e F,\e S}$. Let $f\colon\ut\ra U$ be the canonical map.
The adjoint morphism $\sto\ra f_{*}f^{*}\e\sto$ induces a map $H^{\e
1}_{\text{Nis}}(U,\sto)\ra H^{\e 1}_{\text{Nis}}\be\big(\e
U,f_{*}f^{*}\e\sto\e\big)^{G}$. On the other hand, $H^{\e
1}_{\text{Nis}}\be\big(\e U, f_{*}f^{*}\e\sto\e\big)=H^{\e
1}_{\text{Nis}}\be\big(\e\ut,f^{*}\sto\big)$ by the exactness of
$f_{*}$ for the Nisnevich topology. Further, the base-change
morphism $f^{*}\sto=\sto\times_{U}\e\ut\ra\tto$ induces a morphism
$H^{\e 1}_{\text{Nis}}\be\big(\e\ut,f^{*}\sto\big)\ra H^{\e
1}_{\text{Nis}}\be\big(\ut,\tto\big)$. Thus, there exists a
canonical map
$$
H^{\e 1}_{\text{Nis}}(U,\sto)\ra H^{\e
1}_{\text{Nis}}\be\big(\ut,\tto\big)^{G}.
$$
By Theorem 3.5 below, the preceding map corresponds to a map
\begin{equation}\label{cap}
j_{\e T,\e K/F,\e S}\colon C_{\e T,\e F,\e S}\ra C_{\e T,\e K,\e
S_{K}}^{\e G}
\end{equation}
which is called the $S$-{\it capitulation map}. We will now define a
{\it norm} map $N_{\e T,\e K/F,\e S}\colon C_{\e T,\e K,\e S_{K}}\ra
C_{\e T,\e F,\e S}$, assuming only that $K/F$ is a finite separable
extension.

If $S^{\e\prime}\ra S$ is a morphism of schemes and $\s G$ is an
$S^{\e\prime}$-group scheme, we will write $R_{\e S^{\e\prime}\be/S}(\s
G\e)$ for the Weil restriction of $\s G$ with respect to the given morphism $S^{\e\prime}\ra S$.
See \cite{BLR}, \S7.6, for basic information regarding this construction.

\begin{lemma}  Let $A$ be a commutative ring with unit, let $B$ be a locally
free\footnote{As pointed out by B.Edixhoven, this hypothesis
is needed to ensure the representability of certain functors
considered in \cite{ELL}, proof of Theorem 1.} finitely generated
$A$-module and let $\s G$ be a $B$-group scheme. Assume that $\s G$
is smooth and has geometrically connected fibers. Then $R_{B/A}(\s
G\e)$ has geometrically connected fibers.
\end{lemma}
\begin{proof} (Cf. \cite{ELL}, proof of Theorem 1) Let
$\overline{x}\colon\spec k\ra\spec A$ be a geometric point of $A$, where
$k$ is an algebraically closed field. Then $R_{B/A}(\s
G\e)_{\overline{x}}=R_{B_{k}/k}(\s G_{B_{k}}\e)$, where
$B_{k}=B\otimes_{A}k$. Since $B_{k}$ is finite over $k$, it is
isomorphic to a finite product $\prod B_{j}$ of Artinian local
$k$-algebras $B_{j}$ with residue field $k$, whence $R_{B_{k}/k}(\s
G_{B_{k}}\e)=\prod_{\e j}R_{B_{j}/k}(\s G_{B_{j}}\e)$. We are thus
reduced to the case where $A=k$ and $B$ is an Artinian local
$k$-algebra with residue field $k$. Set $G=R_{B/k}(\s G\e)$,
$n=\text{dim}_{k}(B)$ and let $\mathfrak{m}$ be the maximal ideal of
$B$ (so that $\mathfrak{m}^{n}=0$). There exists a filtration of $G$
by subfunctors $G\supset F^{\e 1}G\supset\dots\supset F^{\e n}G=0$,
where, for any $k$-algebra $C$,
$$
F^{\e i}G(C)=\krn[\e\s G(C\otimes_{k}B)\ra\s
G(C\otimes_{k}(B/\mathfrak{m}^{i}))].
$$
The maps $\s G(C\otimes_{k}B)\ra\s
G(C\otimes_{k}(B/\mathfrak{m}^{i}))$ are surjective for all $i$ and
$C$ by the smoothness of $\s G$, which implies that there exists a
canonical exact sequence of $k$-group schemes
$$
0\ra F^{\e 1}G\ra G\ra\s G_{k}\ra 0.
$$
The group $F^{\e 1}G$ is connected by \cite{ELL}, proof of Theorem
1, and the lemma follows.
\end{proof}

Now the norm map $K^{*}\ra F^{*}$ induces a map $R_{K/F}(T_{K})\ra
T$ which, by the N\'eron mapping property, extends uniquely to a map
of N\'eron-Raynaud models $R_{\e\ut/U}\ng\big(\e\tt\e\big)\ra\s T$
(that $R_{\,\ut/U}\ng\big(\e\tt\e\big)$ is the N\'eron-Raynaud model
of $R_{K/F}(T_{K})$ is shown in \cite{BLR}, Proposition 7.6.6,
p.198). The latter map induces a map $\s N\colon
R_{\,\ut/U}\be\big(\tt\big)^{\circ}\ra\sto$. Now the above lemma and
the argument in \cite{NX}, proof of Lemma 3.1, show that
$R_{\,\ut/U}\be\big(\tt\big)^{\circ}=R_{\,\ut/U}\be\big(\tto\big)$,
whence $\s N$ is a map $R_{\,\ut/U}\be\big(\tto\big)\ra\sto$. On the
other hand, since $R_{\,\ut/U}\be\big(\tto\big)=f_{*}\tto$ and $f_{*}$ is exact for the Nisnevich topology, we
have
$$
H^{\e 1}_{\text{Nis}}\!\left(\e
U,R_{\,\ut/U}\be\big(\tto\big)\right)=H^{\e
1}_{\text{Nis}}\be\big(\e U,f_{*}\tto\big)=H^{\e
1}_{\text{Nis}}\be\big(\e \ut,\tto\big).
$$
Thus $\s N$ induces a map $H^{\e 1}_{\text{Nis}}\be\big(\e
\ut,\tto\big)\ra H^{\e 1}_{\text{Nis}}\be\big(\e U,\sto\big)$ which,
by Theorem 3.5 below, corresponds to a map
\begin{equation}\label{norm map}
N_{\e T,\e K/F,\e S}\colon C_{\e T,\e K,\e S_{K}}\ra C_{\e T,\e F,\e
S}.
\end{equation}
This is the desired norm map.

If $T=G_{m,F}$, we will drop $T$ from some of the above notations,
e.g., $C_{\e F,\e S}$ is $C_{\e T,\e F,\e S}$ when $T=G_{m,F}$.

\begin{remark} The preceding argument shows that,
if $T$ is an $F$-torus and $K/F$ is any finite separable extension,
then the N\'eron-Raynaud $S$-class group of the $F$-torus $R_{\e
K/F}(T_{K})$ is canonically isomorphic to $C_{\e T,\e K,\e S_{K}}$.
\end{remark}

For any finite group $G$, any $G$-module $M$ and any integer $i$,
$\widehat{H}^{i}(G,M)$ will denote the $i$-th Tate $G$-cohomology
group of $M$. Recall that a finite group $G$ is called {\it
metacyclic} if every Sylow subgroup of $G$ is cyclic. The classical
H\"older-Burnside-Zassenhaus theorem asserts that a finite group $G$
is metacyclic if and only if $G$ is a semi-direct product of two
cyclic groups whose orders are relatively prime. See \cite{R},
Theorem 10.26, p.246, and \cite{Z}, Theorem V.3.11, p.175.

Let $X={\rm{Hom}}\!\left(\e T\!\left(\e\fb\e\right),\fc\right)$ be
the $G_{\lbe F}$-module of characters of $T$ and let
$X^{\lbe\vee}=\text{Hom}_{\e\bz}\be(X,\bz)$ be its linear dual.
Recall that a free and finitely generated $G_{\lbe F}$-module $X$ is
said to be a {\it permutation} $G_{\lbe F}$-module if it admits a
$\bz$-basis which is permuted by $G_{\lbe F}$. Then
$X\simeq\oplus_{\e i=1}^{\e m}\e\bz[G_{\lbe F}/H_{i}]$ for some open
subgroups $H_{i}$ of $G_{\lbe F}$ and, consequently, the $F$-torus
$T$ associated to $X$ is isomorphic to $\prod_{\e i=1}^{\e m}\!
R_{L_{i}/F}(\bg_{m,\e L_{i}})$ for some finite separable
subextensions $L_{i}/F$ of $\fb/F$. Such a torus is called {\it
quasi-trivial}. A $G_{\lbe F}$-module $X$ as above is called {\it
invertible} if it is isomorphic to a direct factor of a permutation
$G_{\lbe F}$-module\footnote{The reason for this terminology can be
found in \cite{CTS}, p.178. These modules are also called {\it
permutation projective} in the literature.}. We call an $F$-torus
{\it invertible} if its module of characters is an invertible
$G_{\lbe F}$-module or, equivalently, if it is isomorphic to a
direct factor of a quasi-trivial $F$-torus. The $G_{\lbe F}$-module
$X$ is called {\it flasque} (resp., {\it coflasque}) if
$\widehat{H}^{-1}(H,X)=0$ (resp., $\widehat{H}^{1}(H,X)=0$) for
every open subgroup $H$ of $G_{\lbe F}$. Every invertible $G_{\lbe
F}$-module is both flasque and coflasque. Further, the duality
$X\mapsto X^{\lbe\vee}$ transforms flasque modules into coflasque
ones and preserves invertible modules. An $F$-torus is called {\it
flasque} (resp., {\it coflasque}) if its character module is a
flasque (resp., coflasque) $G_{\lbe F}$-module. The usefulness of
these notions stems from the fact that every $F$-torus $T$ admits
both a flasque resolution and a coflasque resolution, i.e., there
exist a flasque torus $T_{1}$, a coflasque torus $T_{2}$,
quasi-trivial tori $Q_{1}$ and $Q_{2}$ and exact sequences of
$F$-tori
\begin{equation}\label{f}
0\ra T_{1}\ra Q_{1}\ra T\ra 0
\end{equation}
and
\begin{equation}\label{c}
0\ra T\ra Q_{2}\ra T_{2}\ra 0.
\end{equation}
Further, the tori $T_{i}$ and $Q_{i}$ above are determined by $T$ up
to multiplication by a quasi-trivial torus. See \cite{CTS}, Lemma 5,
p.181 or \cite{V}, p.50. By work of S.Endo and T.Miyata \cite{EM},
any $F$-torus which is split by a {\it metacyclic} extension of $F$
admits, in fact, an {\it invertible resolution}, i.e., $T_{1}$ in
\eqref{f} may be taken to be {\it invertible}. See \cite{CTS},
Proposition 2, p.184, or \cite{V}, pp.54-55.

\smallskip

For any abelian group $M$, we will write
$M^{D}=\text{Hom}_{\e\bz}\lbe(M,\bq/\bz)$ for its Pontryagin dual.
If $M$ is finite, $[M]$ will denote the order of $M$.

Finally, recall the group
$$
\Sha^{2}(T)=\krn\!\left[\e H^{2}(F,T)\ra\prod_{\text{all
$v$}}H^{2}(F_{v},T)\right].
$$

\section{Nisnevich cohomology and class sets}

For the convenience of the reader, we present in this Section
Ye.Nisnevich's cohomological interpretation of the class set of an
affine group scheme. Our main reference is \cite{Nis}, Chapter I.

\medskip

If $Y$ is any scheme, $Y_{0}$ will denote the set of closed points
of $Y$. Let $Y$ be a noetherian scheme.

\begin{definition} Let $f\colon Z\ra Y$ be an
\'etale morphism. A point $y\in Y$ is said to {\it split completely}
in $Z$ if there exists a point $z\in f^{-1}(y)$ such that the
induced map of residue fields $k(y)\ra k(z)$ is an isomorphism.
\end{definition}

We denote the set of points of $Y$ which split completely in $Z/Y$
by $\text{cs}(Z/Y)$.

\begin{example} Assume that the set $S$ introduced in the previous Section
contains all prime ideals of $F$ which ramify in $K$. Then a closed
point $v$ of $U$ splits completely in $\ut$ (as defined above) if
and only if the prime ideal of $F$ corresponding to $v$ splits
completely in $K$.
\end{example}

The {\it Nisnevich topology} on $Y$, denoted by $Y_{\text{Nis}}$, is
the Grothendieck topology (\cite{T}, 1.2.1, p.24) whose underlying
category is the category $\text{\'Et}/Y$ of all \'etale $Y$-schemes
and whose set of coverings is defined as follows. A {\it covering}
is a finite family $\{Z_{\e i}\ra Z\}_{\e i\in I}$ of \'etale
morphisms such that $Z=\bigcup_{\e _{i\in I}}\text{cs}(Z_{\e i}/Z)$,
i.e., every point of $Z$ splits completely in some $Z_{i}$. If $Z$
is the spectrum of a Dedekind ring, then any Nisnevich covering
$\{Z_{\e i}\ra Z\}_{\e i\in I}$ must contain an open immersion
$Z_{\e i_{0}}\hookrightarrow Z$ (since the generic point of $Z$
splits completely in some $Z_{\e i}$). We denote this topology by
$Y_{\text{Nis}}$. Clearly, $Y_{\text{Nis}}$ is finer than
$Y_{\text{Zar}}$ but coarser than $Y_{\text{\'et}}$. We will write
$\widetilde{Y}_{\text{Nis}}$ for the category of sheaves of groups
on $Y_{\text{Nis}}$. As is well-known (see, e.g., \cite{T}, Example
6.3.1, p.126), localization in the \'etale topology leads to the
strict henzelization $\s O_{\e Y,\e y}^{\le sh}$ of the local ring
$\s O_{\e Y,\e y}$. By contrast, localization in the Nisnevich
topology leads to the henselization $\s O_{\e Y,\e y}^{\e h}$ of $\s
O_{\e Y,\e y}$ \cite{Nis2}, 1.9.1, p.259. Now, any morphism of
schemes $f\colon Z\ra Y$ induces a morphism of topologies
$f_{\e\text{Nis}}\colon Y_{\text{Nis}}\ra Z_{\e\text{Nis}}$ and
therefore direct and inverse image functors
$$
\begin{array}{rcl}
f_{*}\colon
\widetilde{Z}_{\e\text{Nis}}\ra\widetilde{Y}_{\text{Nis}}\\
f^{*}\colon
\widetilde{Y}_{\text{Nis}}\ra\widetilde{Z}_{\e\text{Nis}}.
\end{array}
$$
As is the case for the \'etale topology, $f_{*}$ is {\it exact} if
$f$ is a finite morphism. This may be proved by imitating the proof
of the corresponding fact for the \'etale topology \cite{T}, proof
of Theorem 6.4.2, p.129, using \cite{Nis2}, Lemma 1.18.1, p.268, in
place of \cite{T}, Lemma 6.2.3, p.124.

If $V$ is a nonempty open subscheme of $U$, the ring of $V$-integral
adeles of $U$ is defined by
$$
{\mathbb A}_{\e U}(V)=\displaystyle\prod_{v\e\in\e U\setminus
V}F_{v}\times\displaystyle\prod_{v\e\in\e V_{0}}{\s O}_{v}.
$$
Now define a partial ordering on the family of all nonempty open
subschemes of $U$ by setting $V\leq V^{\e\prime}$ if
$V^{\e\prime}\subset V$. Then, for every pair $V,V^{\e\prime}$ of
nonempty open subschemes of $U$ such that $V\leq V^{\e\prime}$,
there exists a canonical map ${\mathbb A}_{\e U}(V)\ra {\mathbb
A}_{\e U}(V^{\e\prime})$. The ring of adeles of $U$ is by definition
$$
{\mathbb A}_{\e U}=\varinjlim_{V}\e{\mathbb A}_{\e
U}(V).\bibliographystyle{mrl}
$$
Let $\s H$ be a generically smooth $U$-group scheme of {\it finite
type}. If $U^{\e\prime}=\spec\s O^{\e\prime}$ is an affine \'etale
$U$-scheme with fraction field $F^{\e\prime}$, we let
$U^{\e\prime}_{*}$ denote the normalization of $U$ in
$F^{\e\prime}$. Define a sheaf $\widehat{\s H}$ on $U_{\text{Nis}}$
by
$$
\widehat{\s H}(U^{\e\prime}):=\s H({\mathbb A}_{\e
U^{\e\prime}_{\lbe
*}}(U^{\e\prime}))=\displaystyle\prod_{v^{\e\prime}\e\in\e
U^{\e\prime}_{\lbe *}\e\setminus\e U^{\e\prime}}\s H(\e
F^{\e\prime}_{\be
v^{\e\prime}}\e)\times\displaystyle\prod_{v^{\e\prime}\e\in\e
U^{\e\prime}_{0}}\s H({\s O}^{\e\prime}_{ v^{\e\prime}}).
$$
Then $\widehat{\s H}(U^{\e\prime})=\prod_{\, v\in U_{0}}\s H(\e{\s
O}_{v}\otimes_{\e\s O_{F,\e S}}\s O^{\e\prime}\e)$, which yields the
following alternative description of $\widehat{\s H}$:
\begin{equation}\label{n}
\widehat{\s H}=\displaystyle\prod_{v\in U_{0}}(j_{v})_{*}\e
j_{v}^{*}\,\widehat{\s H},
\end{equation}
where, for each $v\in U_{0}$, $j_{v}\colon\spec{\s O}_{v}\ra U$ is
the canonical morphism.

\begin{lemma} $H_{{\rm{Nis}}}^{1}\big(U,\widehat{\s H}\e\big)=0$.
\end{lemma}
\begin{proof} By \eqref{n}, it suffices to check that
$H_{{\rm{Nis}}}^{1}\big(U,(j_{v})_{*}\e j_{v}^{*}\,\widehat{\s
H}\e\big)=0$ for every $v\in U_{0}$. The pointed set
$H_{{\rm{Nis}}}^{1}\big(U,(j_{v})_{*}\e j_{v}^{*}\,\widehat{\s
H}\e\big)$ injects into $H_{{\rm{Nis}}}^{1}\big({\s
O}_{v},j_{v}^{*}\,\widehat{\s H}\e\big)$, which is trivial since a
complete local scheme does not have nontrivial coverings in the
Nisnevich topology.
\end{proof}

The diagonal embedding $\s O^{\e\prime}\ra {\mathbb A}_{\e
U^{\e\prime}_{*}}(U^{\e\prime})$ (where $U^{\e\prime}=\spec\s
O^{\e\prime}$ is any affine \'etale $U$-scheme) induces an injection
of Nisnevich sheaves $\s H\hookrightarrow \widehat{\s H}$. Let $\s
Q=\widehat{\s H}/\s H$ be the corresponding quotient sheaf. The
stalk of $\s Q$ at the generic point of $U$ is the group $\s
H({\mathbb A}_{\e U})/\s H(F)$.

\begin{lemma} The canonical map $\s Q(U)\ra\s H({\mathbb A}_{\e
U})/\s H(F)$ is a bijection.
\end{lemma}
\begin{proof} To prove injectivity, assume that $q^{1}$ and $q^{2}\in \s Q(U)$
have the same image under the above map. There exists a covering
$\{U_{i}\ra U\}_{i\in I}$ and families of sections $s_{i}^{\le k}\in
\widehat{\s H}(U_{i})$ and $s_{ij}^{\le k}\in\widehat{\s
H}(U_{i}\times_{U}U_{j})$ ($k=1,2$) such that $s_{i}^{\e
k}=s_{j}^{\e k}s_{ij}^{\e k}$ for all $i,j,k$ and $p(s_{i}^{\le
k})=q_{i}^{\e k}$ for all $i,k$, where $q_{i}^{\e k}$ is the
restriction of $q^{\e k}$ to $U_{i}$ and $p\colon\widehat{\s H}\ra\s
Q$ is the canonical projection. The fact that $q^{1}$ and $q^{2}$
have the same image in $\s H({\mathbb A}_{\e U})/\s H(F)$ under the
map of the lemma means that there exist an index $i_{0}\in I$, a
nonempty Zariski-open subset $U_{i_{0}}\hookrightarrow U$ and a
section $g\in \s H(F)$ such that $s_{i_{0}}^{\e 1}=s_{i_{0}}^{\e
2}g$. Then, for every $i$,
$$
s_{i}^{\e 1}=s_{i_{0}}^{\e 1}\e s_{i,i_{0}}^{\e 1}=s_{i_{0}}^{\e
2}\e g\e s_{i,i_{0}}^{\e 1}=s_{i}^{\e 2}\e s_{i_{0},i}^{\e 2}\e g\e
s_{i,i_{0}}^{\e 1}.
$$
It follows from the above that $g\in \widehat{\s
H}\e(U_{i}\times_{U}U_{i_{0}})\cap\s H(F)\subset{\s
H}(U_{i}\times_{U}U_{i_{0}})$. Consequently,
$$
(s_{i}^{\e 2})^{-1}s_{i}^{\e 1}=s_{i_{0},i}^{\e 2}\e g\e
s_{i,i_{0}}^{\e 1}\in\widehat{\s H}\e(U_{i})\cap{\s
H}(U_{i}\times_{U}U_{i_{0}})={\s H}\e(U_{i}),
$$
whence $s_{i}^{\e 1}=s_{i}^{\e 2}\e{\s H}\e(U_{i})$ for all $i$. We
conclude that
$$
q_{i}^{\e 1}=p(s_{i}^{\e 1})=p(s_{i}^{\e 2})=q_{i}^{\e 2}
$$
for all $i$, whence $q^{1}=q^{2}$.

 To prove surjectivity, let $c\in
\s H({\mathbb A}_{\e U})/\s H(F)$ and let $x=(x_{v})_{v\in
U_{0}}\in\s H({\mathbb A}_{\e U})$ be a representative of $c$. There
exists a Zariski open subscheme $U_{i_{0}}\subset U$ such that
$x\in\s H({\mathbb A}_{\e U}(U_{i_{0}}))$. We write $f_{i_{0}}$ for
the canonical inclusion $U_{i_{0}}\hookrightarrow U$. Let
$S^{\e\prime}=U\setminus U_{i_{0}}$. Since $\s H_{F}$ is smooth,
R.Elkik's approximation theorem \cite{El}, \S II, shows that $\s
H(\s O_{v}^{\e h})$ (resp., $\s H(F_{v}^{\e h})$) is dense in $\s
H({\s O}_{v})$ (resp., $\s H(F_{v})$) in the $v$-adic topology for
every $v\in U_{0}$. Since $\s H({\s O}_{v})$ is open in $\s
H(F_{v})$, we conclude that
$$
\s H(F_{v})=\s H({\s O}_{v})\s H(F_{v}^{\e h})
$$
for every $v\in U_{0}$. Consequently, for each $v\in S^{\e\prime}$,
$x_{v}\in \s H(F_{v})$ decomposes as $x_{v}=a_{\le v}b_{\le v}$,
where $a_{\le v}\in \s H({\s O}_{v})$ and $b_{\le v}\in\s
H(F_{v}^{\e h})$. Choose, for each $v\in S^{\e\prime}$, a finite
extension $F_{v}^{\e\prime}$ of $F_{v}$ contained in $F_{v}^{\e h}$
such that $b_{\le v}\in\s H(F_{v}^{\e\prime})$. Further, define
$U_{i_{0},v}= U_{i_{0}}\cup\{v\}$ and let $U_{i_{v}}=\spec\s
O_{i_{v}}$ be the normalization of $U_{i_{0},v}$ in
$F_{v}^{\e\prime}$. Note that $\s H(U_{i_{v}}\times_{U}\spec F)=\s
H(F_{v}^{\e\prime}\e)$. We write $f_{i_{v}}\colon U_{i_{v}}\ra U$
for the canonical morphism. Let
$x^{\e\prime}=(x^{\e\prime}_{v})_{v\in U_{0}}$ be the adele given by
$$
x^{\e\prime}_{v}=\begin{cases}a_{v}\quad\text{if $v\in S^{\e\prime}$}\\
x_{v}\quad\text{if $v\in U_{i_{0}}$}.\end{cases}
$$
Clearly, $x\equiv x^{\e\prime}\pmod{\s H(F_{v}^{\e\prime})}$ for
every $v\in S^{\e\prime}$. Now let $I=\{i_{0}\}\cup\{i_{v}\colon
v\in S^{\e\prime}\}$ and let
$$\begin{array}{rcl}
q_{\e i_{v}}&=&x^{\e\prime}\e\s H(U_{i_{v}})\in \s Q(U_{i_{v}})
\quad(v\in S^{\e\prime})\\
q_{\e i_{0}}&=&x\s H(U_{i_{0}})\in\s Q(U_{i_{0}})
\end{array}
$$
be a family of local sections of $\s Q$ associated to the covering
$(f_{i}\colon U_{i}\ra U)_{i\in I}$. Then $(q_{i})_{i\in I}$ defines
a section $q\in\s Q(U)$ which maps to $c$.
\end{proof}

The above lemma enables us to identify $\s Q(U)$ and $\s H({\mathbb
A}_{\e U})/\s H(F)$. The map $p$ appearing in the proof of the lemma
induces a map $p_{_{U}}\colon \s H\!\left(\lbe{\mathbb A}_{\e
U}(U)\lbe\right)\ra\s Q(U)$. The {\it class set} $\e C(\s H)$ of $\s
H$ is by definition the coset space $\img p_{_{U}}\!\setminus\be\s
Q(U)$, i.e.,
$$
C(\s H)=\e p_{_{U}}\be(\s H\!\left(\lbe{\mathbb A}_{\e
U}(U)\lbe\right))\backslash\e \s H\!\left(\lbe{\mathbb A}_{\e
U}\lbe\right)\be/\e\s H(\lbe F\lbe).
$$

\begin{theorem} There exists a canonical bijection
$$
\delta_{\e\s H}\colon C(\s H)\overset{\sim}\ra
H_{{\rm{Nis}}}^{1}(U,{\s H}\e).
$$
\end{theorem}
\begin{proof} By Lemma 3.3, the exact sequence of Nisnevich sheaves
$$
1\ra\s H\ra\widehat{\s H}\ra\s Q\ra 1
$$
induces an exact sequence of pointed sets
$$
1\ra\s H(U)\ra \widehat{\s H}(U)\overset{p_{_U}}\longrightarrow\s
Q(U)\ra H_{{\rm{Nis}}}^{1}(U,{\s H}\e)\ra 1.
$$
Thus there exists a bijection $C(\s H)=\img p_{_U}\!\setminus\be\s
Q(U)\overset{\sim}\ra H_{{\rm{Nis}}}^{1}(U,{\s H}\e)$.
\end{proof}

\begin{remark} The proof of Lemma 3.4 shows why the Zariski topology
is too coarse to yield a cohomological interpretation of the class
set $C(\s H)$: in general, $\s H(\s O_{(v)})$, where $\s O_{(v)}$
denotes the local ring of $U$ at $v$, is not dense in $\s H({\s
O}_{v})$ (failure of weak approximation). On the other hand, the
\'etale topology is too fine, in the sense that
$H_{\text{\'et}}^{1}(U,{\s H}\e)$ is usually larger than $C(\s H)$
(in fact, it can be shown that there exists a canonical embedding
$C(\s H)\hookrightarrow H_{\text{\'et}}^{1}(U,{\s H}\e)$ whose
cokernel is often nontrivial).
\end{remark}

\section{The capitulation kernel}

In this Section we prove Theorem 1.1 (this is Theorem 4.6 below).

\smallskip

For each prime $w$ of $K$, the canonical map $\tt_{\e\s O_{
w}}\ra(i_{w})_{*}\Phi_{w}$ induces a map $\vartheta_{\lbe w}\colon
T(K_{w})\ra\Phi_{w}(k(w))$ which generalizes the $w$-adic valuation
$\text{ord}_{\le w}\colon K_{w}^{*}\ra\bz$\footnote{\e Keeping this
in mind while reading Sections 4 and 5 should serve as a guide for
the reader.}. The composite $T(K)\hookrightarrow
T(K_{w})\overset{\vartheta_{\lbe w}}\longrightarrow \Phi_{w}(k(w))$
will also be denoted by $\vartheta_{\lbe w}$. For each $v\notin S$,
we have a canonical map $\oplus_{w\mid v}\vartheta_{\lbe w}\colon
T(K)\ra\bigoplus_{w\mid v}\Phi_{w}(k(w))$. Consider
\begin{equation}\label{theta}
\vartheta_{_{\be S}}=\bigoplus_{v\notin S}\bigoplus_{w\mid
v}\vartheta_{\lbe w}\colon T(K)\ra\bigoplus_{v\notin
S}\bigoplus_{w\mid v}\Phi_{w}(k(w)).
\end{equation}

\begin{proposition} There exists a canonical exact sequence of
$G$-modules
$$
1\ra\tto\be\big(\ut\big)\ra T(K)\overset{\vartheta_{_{\!
S}}}\longrightarrow\displaystyle\bigoplus_{v\notin
S}\bigoplus_{w\mid v}\,\Phi_{w}(k(w))\ra  C_{\e T,K,S_{K}}\ra 0,
$$
where $\vartheta_{_{\! S}}$ is the map \eqref{theta}.
\end{proposition}
\begin{proof} See \cite{mrl}, \S 3.
\end{proof}

We now split the exact sequence of the proposition into two short
exact sequences of $G$-modules as follows:
\begin{equation}\label{seq1}
1\ra\tto\be\big(\ut\big)\ra T(K)\ra T(K)/\,\tto\be\big(\ut\big)\ra 1
\end{equation}
and
$$
1\ra T(K)/\,\tto\be\big(\ut\big)\ra\displaystyle\bigoplus_{v\notin
S}\bigoplus_{w\mid v}\,\Phi_{w}(k(w))\ra C_{\e T,\e K,\e S_{K}}\ra
0. $$
These sequences induce connecting homomorphisms
\begin{equation}\label{p1}
\partial_{\e 1}\colon\big(T(K)/\tto\be\big(\ut\big)\big)^{G}\ra H^{1}\ng
\big(G,\tto\be(\ut)\big)
\end{equation}
and
\begin{equation}\label{p2}
\partial_{\e 2}\colon C_{\e T,\e K,\e S_{K}}^{\e G}\ra H^{\e
1}\big(G,T(K)/\,\tto\be\big(\ut\big)\big).
\end{equation}
For a general description of these homomorphisms, see \cite{AW},
p.97. Set
$$
(C_{\e T,\e K,\e S_{K}})_{\text{trans}}^{G}=\krn(\partial_{\e 2}\e).
$$
From the general description of $\partial_{\e 2}$ just mentioned, it
is not difficult to check that the image of the capitulation map
$j_{\e T,\e K/F,\e S}\colon C_{\e T,\e F,\e S}\ra C_{\e T,\e K,\e
S_{K}}^{\e G}$ is contained in $(C_{\e T,\e K,\e
S_{K}})_{\text{trans}}^{G}$. Thus $j_{\e T,\e K/F,\e S}$ induces a
map
\begin{equation}\label{j'}
j_{_{T,\e K/F,S}}^{\e\prime}\colon C_{\e T,F,S}\ra (C_{\e
T,K,S_{K}})_{\text{trans}}^{G}
\end{equation}
such that $\krn j_{\e T,\e K/F,\e S}^{\e\prime}=\krn j_{_{T,\e
K/F,S}}$.

For each $v\notin S$, we will identify the $G$-module
$\bigoplus_{w\mid v}\Phi_{w}(k(w))$ with the $G$-module induced by
the $G_{w_{v}}$-module $\Phi_{w_{v}}\be(k(w_{v}))$, where $w_{v}$ is
the prime of $K$ lying above $v$ fixed previously and
$G_{w_{v}}=\text{Gal}(K_{w_{v}}/F_{v})$. By Shapiro's lemma,
$$
H^{\e i}\big(G,\textstyle\bigoplus_{w\mid
v}\Phi_{w}(k(w))\big)=H^{\e
i}(G_{w_{v}},\Phi_{w_{v}}\be(k(w_{v})))
$$
for every $i\geq 0$. Further, since $I_{w_{v}}$ acts trivially on
$\Phi_{w_{v}}\be(k(w_{v}))$, we have
$$
\Phi_{w_{v}}\be(k(w_{v}))^{G_{w_{v}}}=
\Phi_{w_{v}}\be(k(w_{v}))^{G(w_{v})},
$$
where $G(w_{v})=\text{Gal}(k(w_{v})/k(v))$. Thus
$$
\big(\textstyle\bigoplus_{w\mid v}\Phi_{w}(k(w))\big)^{\be
G}=\Phi_{w_{v}}\be(k(w_{v}))^{G(w_{v})}.
$$
There exists a canonical map
$\Phi_{v}(k(v))\ra\big(\textstyle\bigoplus_{w\mid
v}\Phi_{w}(k(w))\big)^{\be G}$, and therefore we obtain a map
\begin{equation}\label{loc cap}
\delta_{v}=\delta_{\e T,\e K/F,\e v}\colon\Phi_{v}(k(v))\ra
\Phi_{w_{v}}\be(k(w_{v}))^{G(w_{v})}.
\end{equation}

We will write $B$ for the set of non-archimedean primes of $F$ where
$T$ has bad reduction.

\begin{lemma} There exists a canonical isomorphism
$$
\krn(\oplus_{v\notin S}\,\delta_{v})=\bigoplus_{v\e\in\e B\e\setminus\e
S}H^{\e 1}\!\be\left(I_{w_{v}},T\!\left(K_{w_{v}}
^{\le{\rm{nr}}}\right)\right)^{G_{k(\lbe v)}}.
$$
\end{lemma}
\begin{proof} By \cite{mrl}, Lemma 3.3, for each $v\notin S$ there exists a
canonical isomorphism $\krn\delta_{v}=H^{\e
1}\!\left(I_{w_{v}},T\!\left(K_{w_{v}}
^{\le{\rm{nr}}}\right)\right)^{G_{k(\lbe v)}}$. Now, if $T$ splits
over $F_{v}^{\e\text{nr}}$ (i.e., has multiplicative reduction at
$v$ \cite{NX}, Proposition 1.1), then $\krn\delta_{v}=0$ by
Hilbert's theorem 90. This yields the lemma.
\end{proof}

\begin{remark} As shown in \cite{mrl}, proof of Lemma 3.3, $H^{\e 1}\!\be\left(I_{w_{v}},T\!\left(K_{w_{v}}
^{\le{\rm{nr}}}\right)\right)^{G_{k(\lbe v)}}$ is canonically isomorphic to a subgroup of $\Phi_{v}(k(v))_{\text{tors}}$. It follows that
$\bigoplus_{\e v\in B\setminus
S}H^{\e 1}\!\be\left(I_{w_{v}},T\!\left(K_{w_{v}}
^{\le{\rm{nr}}}\right)\right)^{G_{k(\lbe v)}}=0\,$ if $\,\Phi_{v}(k(v))_{\text{tors}}=0$ for every $v\in B\setminus S$.
\end{remark}

\medskip

Next we describe $\cok\be(\oplus_{v\notin S}\,\delta_{v})$ under the
assumption that $T_{K}$ has multiplicative reduction over $\ut$,
i.e., $S_{K}$ contains all primes of bad reduction for $T_{K}$ or,
equivalently, $I_{\overline{w}}$ acts trivially on $X$ for every
$w\notin S_{K}$ (see \cite{NX}, Proposition 1.1)\footnote{The
argument that follows is a generalization of the proof of
\cite{CAmb}, Lemma 2.3.}. Note that the above assumption is
trivially satisfied if $K$ splits $T$. In general,
$\cok\be(\oplus_{v\notin S}\,\delta_{v})$ fits into a 5-term exact
sequence (which is omitted).

Let $v\notin S$. The inertia group $I_{\e\overline{v}}$ acts on the
$G_{\lbe F}$-module $X$ through a finite quotient $J_{\overline{v}}$
(say) and there exists a canonical map
$$
{\rm{Nm}}_{v}\colon X\ra
X^{I_{\overline{v}}},\chi\mapsto\sum_{g\e\in
J_{\overline{v}}}\chi^{\e g}.
$$
Let $\widehat{T}_{v}$ be the $F_{v}$-torus which corresponds to the
subgroup $\krn{\rm{Nm}}_{v}$ of $X$. Then there exists a canonical
exact sequence
$$
0\ra T^{\e\prime}_{v}\ra T_{F_{v}}\ra \widehat{T}_{\lbe v}\ra 0,
$$
where $T^{\e\prime}_{v}$ is the largest subtorus of $T_{F_{v}}$
having multiplicative reduction. See \cite{NX}, Proposition 1.2. In
particular, $\widehat{T}_{\lbe v}=0$ if $T$ has multiplicative
reduction at $v$. We will write $\Phi^{\e\prime}_{v}$ (resp.,
$\widehat{\Phi}_{v}$) for the sheaf of connected components of the
N\'eron-Raynaud model of $T^{\e\prime}_{v}$ (resp.,
$\widehat{T}_{\lbe v}$). Let
\begin{equation}\label{ct}
C_{\e T,\e F_{v}}=\cok[\e
T(F_{v})\overset{\vartheta_{\lbe v}}\longrightarrow\Phi_{v}(k(v))],
\end{equation}
where $\vartheta_{\lbe v}$ is the generalization of $\text{ord}_{\le
v}$ introduced above. There exists a canonical exact commutative
diagram
$$
\xymatrix{0\ar[r] &
T^{\e\prime}_{v}(F_{v})\ar[d]^{\vartheta_{\be v}^{\e\prime}}\ar[r] &
T(F_{v})\ar[r]\ar[d]^{\vartheta_{\lbe v}}&\widehat{T}_{\lbe v}(F_{v})
\ar[d]^{\widehat{\vartheta}_{v}}&\\
0\ar[r] & \Phi^{\e\prime}_{v}(k(v))\ar[r]&
\Phi_{v}(k(v))\ar[r]&\widehat{\Phi}_{v}(k(v))\ar[r]&0. }
$$
(For the exactness of the bottom row, see \cite{BL}, p.288, line 2).
By \cite{BL}, Proposition 3.2, the map $\vartheta_{\be
v}^{\e\prime}$ is surjective. We conclude that there exists a
canonical isomorphism
$$
C_{\e T,\e F_{v}}=\cok[\e
T(F_{v})\ra\widehat{\Phi}_{v}(k(v))\e],
$$
where the map involved is the composite $T(F_{v})\ra
\widehat{T}_{\lbe
v}(F_{v})\overset{\widehat{\vartheta}_{v}}\longrightarrow
\widehat{\Phi}_{v}(k(v))$. In particular, $C_{\e T,\e F_{v}}=0$
if $v\notin B$.

Now, since $T_{K}$ has multiplicative reduction at $w_{v}$ by
assumption, the map $\vartheta_{\lbe w_{v}}\colon T(K_{w_{v}})\ra
\Phi_{w_{v}}(k(w_{v}))$ is surjective. It follows that there exists
a canonical exact commutative diagram\footnote{The commutativity of
the left-hand rectangle generalizes the formula ``$\,\text{ord}_{\e
w_{v}}\!(x)=e_{v}\e\text{ord}_{\e v}(x)$ ($x\in F_{v}^{\e *}\,$)".
Compare \eqref{d1} with the diagram in \cite{CAmb}, proof of Lemma
2.3.}
\begin{equation}\label{d1}
\xymatrix{T(F_{v})\ar[r]^{\vartheta_{\lbe v}}\ar@{=}[d]&\Phi_{v}(k(v))\ar[d]^{\delta_{v}}\ar[r]&
C_{\e T,\e F_{v}}\ar[r]\ar[d]&0\\
T(F_{v})\ar[r]^(.38){\tilde{\vartheta}_{\lbe
w_{v}}}&\Phi_{w_{v}}\be(k(w_{v}))^{G(w_{v})} \ar[r]^(.45){\partial_{
w_{v}}} & H^{1}(G_{w_{v}},\tto(\s O_{w_{v}}))^{\e\prime}\ar[r]&0,}
\end{equation}
where the bottom row is part of the $G_{w_{v}}$-cohomology sequence
induced by the exact sequence $0\ra \tto(\s O_{w_{v}})\ra
T(K_{w_{v}})\overset{\vartheta_{ w_{\lbe v}}}\longrightarrow
\Phi_{w_{v}}\be(k(w_{v}))\ra 0$ and
\begin{equation}\label{h1}
H^{1}(G_{w_{v}},\tto(\s
O_{w_{v}}))^{\e\prime}:=\krn\!\be\left[H^{1}(G_{w_{v}},\tto(\s
O_{w_{v}}))\ra H^{1}(G_{w_{v}},T(K_{w_{v}}))\right].
\end{equation}
The map $\tilde{\vartheta}_{\lbe w_{v}}$ is the restriction of
$\vartheta_{\lbe w_{v}}$ to $T(F_{v})\subset T(K_{w_{v}})$ and
$\partial_{\e w_{v}}$ is induced by the connecting homomorphism
$\Phi_{w_{v}}\be(k(w_{v}))^{G(w_{v})}\ra H^{1}(G_{w_{v}},\tto(\s
O_{w_{v}}))$. Let
\begin{equation}\label{cbar}
\overline{C}_{\e T,\e F_{v}}=\cok[\,\krn\delta_{v}\ra C_{\e T,\e
F_{v}}\e],
\end{equation}
where the map involved is induced by the projection
$\Phi_{v}(k(v))\ra C_{\e T,\e F_{v}}$. Applying the
snake lemma to the diagram which is derived from \eqref{d1} by
replacing both instances of $T(F_{v})$ there by their images in
their target groups, we obtain the following generalization of
\cite{CAmb}, Lemma 2.3.

\begin{proposition} Assume that $T_{\be K}$ has
multiplicative reduction over $\ut$. Then there exists a canonical
exact sequence
$$
0\ra\displaystyle\bigoplus_{v\in B\e\setminus\e S}\!\overline{C}_{\e T,\e F_{v}}\ra \displaystyle\bigoplus_{v\notin S}H^{1}(G_{w_{v}},\tto(\s
O_{w_{v}}))^{\e\prime}\overset{\oplus\e \psi_{v}}\longrightarrow\displaystyle\bigoplus_{v\notin S}\cok\delta_{v}\ra
0,
$$
where the groups $\overline{C}_{\e T,\e F_{v}}$ and
$H^{1}(G_{w_{v}},\tto(\s O_{w_{v}}))^{\e\prime}$ are given by
\eqref{cbar} and \eqref{h1} and the map $\psi_{v}$ is defined as
follows: if $\e\xi_{w_{v}}\in H^{1}(G_{w_{v}},\tto(\s
O_{w_{v}}))^{\e\prime}$, then $\psi_{v}(\xi_{w_{v}})$ is the element
of $\e\cok\delta_{v}$ which is represented by any element of
$\e\partial_{\e w_{v}}^{\e -1}(\xi_{w_{v}})$, where $\partial_{\e
w_{v}}$ is the connecting homomorphism appearing in diagram
\eqref{d1}.\qed
\end{proposition}

Now there exists an exact commutative diagram
\begin{equation}\label{r}
\xymatrix{0\ar[r] & T(F)/\e{\s T}^{\e\circ}(U)\ar[d]^{\gamma}\ar[r]
& \displaystyle\bigoplus_{v\notin
S}\Phi_{v}(k(v))\ar@{->>}[r]\ar[d]^{\bigoplus_{v\notin
S}\delta_{v}} &
C_{\e T,F,S}\ar[d]^{j_{\e T,K/F,S}^{\e\prime}}\\
0\ar[r]
&\big(T(K)/\e\tto\be\big(\ut\big)\big)^{G}\ar[r]^(.47){\overline{\vartheta}_{_{\be S}}}
& \displaystyle\bigoplus_{v\notin
S}\Phi_{w_{v}}\be(k(w_{v}))^{G(w_{v})}\ar@{->>}[r] & (C_{\e
T,K,S_{K}}^{\e G})_{{\text{trans}}}}
\end{equation}
where $\gamma$ is induced by the inclusion $T(F)\hookrightarrow
T(K)$, $j_{\e T,K/F,S}^{\e\prime}$ is the map \eqref{j'} and
$\overline{\vartheta}_{_{\be S}}$ is induced by \eqref{theta}. By
\cite{mrl}, proof of Lemma 3.7,
$$
\krn\gamma=\e\tto\be\big(\ut\big)^{G}\be/{\s T}^{\e\circ}(U).
$$
Thus, recalling that $\krn j_{\e T,K/F,S}^{\e\prime}=\krn j_{\e
T,K/F,S}$ and using Lemma 4.2, the snake lemma applied to \eqref{r}
yields both an exact sequence
\begin{equation}\label{e1}
\begin{array}{rcl}
0\ra\e\tto\be\big(\ut\big)^{G}\be/{\s T}^{\e\circ}(U)\ra\displaystyle\bigoplus_{v\in B\e\setminus\e
S}H^{\e 1}\!\be\left(I_{w_{v}},T\!\left(K_{w_{v}}
^{\le{\rm{nr}}}\right)\right)^{G_{k(\lbe v)}}\!\! &\ra&\!\krn j_{\e T,K/F,S}\\
&\ra&\krn\overline{\vartheta}_{_{\be S}}\ra 0
\end{array}
\end{equation}
and an isomorphism
\begin{equation}\label{cok}
\cok\,\overline{\vartheta}_{_{\be S}}^{\,\prime}=\cok j_{\e T,\e
K/F,\e S}^{\,\prime}\,,
\end{equation}
where
\begin{equation}\label{p}
\overline{\vartheta}_{_{\be S}}^{\,\prime}\colon\cok\gamma\ra\displaystyle
\bigoplus_{v\notin S}\cok\delta_{v}.
\end{equation}
is induced by $\overline{\vartheta}_{_{\be S}}$.

To describe \eqref{p}, we first note that \eqref{p1} induces an
isomorphism
$$
\overline{\partial}_{\e
1}\colon\cok\gamma\overset{\sim}\longrightarrow
H^{1}\ng\big(G,\tto\be(\ut)\big)^{\e\prime}
$$
where
$$
H^{1}\ng\big(G,\tto\be(\ut)\big)^{\e\prime}:=\krn\!\left[H^{1}\ng
\big(G,\tto\be(\ut)\big)\ra H^{1}(G,T(K))\right].
$$
Let
\begin{equation}\label{loc}
\lambda_{S}\colon\displaystyle\bigoplus_{v\notin S}H^{1}\ng\big(G,\tto\be(\ut)\big)^{\e\prime}\ra\displaystyle\bigoplus_{v\notin S}
H^{1}(G_{w_{v}},\tto(\s O_{w_{v}}))^{\e\prime}
\end{equation}
be the natural localization map. Then there exists a commutative diagram
$$
\xymatrix{\cok\gamma\ar[d]^(.45){\overline{\partial}_{\e
1}}_(.45){\simeq}\ar[r]^(.45){\overline{\vartheta}_{_{\be S}}^{\,\prime}}&
\displaystyle\bigoplus_{v\notin S}\cok\delta_{v}\\
H^{1}\ng\big(G,\tto\be(\ut)\big)^{\e\prime}\ar[r]^(.4){\lambda_{S}}&
\displaystyle\bigoplus_{v\notin S} H^{1}(G_{w_{v}},\tto(\s
O_{w_{v}}))^{\e\prime}\ar@{->>}[u]_{\oplus_{v\notin S}\,\psi_{v}},}
$$
where the maps $\psi_{v}$ are defined in the statement of
Proposition 4.4. In other words, $\overline{\vartheta}_{_{\be
S}}^{\,\prime}=\oplus_{v\notin S}\e \psi_{
v}\e\circ\e\lambda_{S}\e\circ\e\overline{\partial}_{\e 1}$. This may
be checked by using the definitions of $\lambda_{S}$ and $\psi_{v}$
and the general description of the connecting homomorphisms
$\partial_{\e w_{v}}$ (from diagram \eqref{d1}) and \eqref{p1}. Now
write
$$
\krn\lambda_{S}=\Sha^{\e
1}_{S}\big(G,\tto\be(\ut)\big)^{\prime}
$$
and
$$
\cok\lambda_{S}=\Tha^{1}_{S}\big(G,\tto\be(\ut)\big)^{\e\prime}.
$$
Then
$$
\Sha^{\e
1}_{S}\big(G,\tto\be(\ut)\big)^{\prime}=\krn\!\!\left[\!\Sha^{\e
1}_{S}\big(G,\tto\lbe(\ut)\big)\ra\Sha^{\e
1}_{S}\big(G,T(K))\right],
$$
where
$$
\Sha^{\e
1}_{S}\big(G,\tto\lbe(\ut)\big)=\krn\!\!\left[\,H^{1}\ng\big(G,\tto\lbe(\ut)\big)\ra
\displaystyle\bigoplus_{v\notin S} H^{1}(G_{w_{v}},\tto(\s
O_{w_{v}}))\right]
$$
and
$$
\Sha^{\e 1}_{S}\big(G,T(K))=\krn\!\!\left[\,H^{1}(G,T(K))\ra
\displaystyle\bigoplus_{v\notin S}
H^{1}(G_{w_{v}},T(K_{w_{v}}))\right].
$$
Applying the snake lemma to the exact commutative diagram
$$
\xymatrix{0\ar[d]\ar[r]&H^{1}\ng\big(G,\tto\lbe(\ut)\big)^{\e\prime}\ar[d]^(.4){\lambda_{S}}\ar[r]^(.6){\overline{\partial}_{1}^{\,-1}}&
\cok\gamma\ar[d]^(.4){\overline{\vartheta}_{_{\be S}}^{\,\prime}}\ar[r]&0\\
\displaystyle\bigoplus_{v\in B\e\setminus\e S}\!\overline{C}_{\e
T,\e F_{v}}\ar@{^{(}->}[r]&\displaystyle\bigoplus_{v\notin S}
H^{1}(G_{w_{v}},\tto(\s
O_{w_{v}}))\ar[r]^(.6){\oplus\e\psi_{v}}&\displaystyle\bigoplus_{v\notin
S}\cok\delta_{v}\ar[r]&0}
$$
and using \eqref{cok}, we obtain the following exact sequence
\begin{equation}
\begin{array}{rcl}\label{e2}
0\ra\Sha^{\e 1}_{S}\big(G,\tto\lbe(\ut)\big)^{\prime}\ra
\krn\overline{\vartheta}_{_{\be S}}^{\,\prime}\ra\displaystyle
\bigoplus_{v\in B\e\setminus\e S}\!\overline{C}_{\e T,\e
F_{v}}&\!\!\ra\!\!&\!\!\!\!\Tha^{\e
1}_{S}\big(G,\tto\lbe(\ut)\big)^{\prime}\\
&\ra&\!\!\!\cok j_{\e T,K/F,S}^{\e\prime}\ra 0.
\end{array}
\end{equation}
We collect together \eqref{e1} and \eqref{e2} in the following
statement.
\begin{proposition} Assume that $T_{K}$ has multiplicative reduction over $\ut$. Then there exist canonical exact sequences
$$
\begin{array}{rcl}
0\ra\e\tto\lbe\big(\ut\big)^{G}\be/{\s
T}^{\e\circ}(U)\ra\displaystyle\bigoplus_{v\in B\e\setminus\e
S}H^{\e 1}\!\be\left(I_{w_{v}},T\!\left(K_{w_{v}}
^{\le{\rm{nr}}}\right)\right)^{G_{k(\lbe v)}}\!\! &\ra&\!\krn j_{\e T,\e K/F,\e S}\\
&\ra&\krn\overline{\vartheta}_{_{\be S}}^{\,\prime}\ra 0
\end{array}
$$
and
$$
\begin{array}{rcl}
0\ra\Sha^{\e 1}_{S}\big(G,\tto\lbe(\ut)\big)^{\prime}\ra
\krn\overline{\vartheta}_{_{\be S}}^{\,\prime}\ra\displaystyle
\bigoplus_{v\in B\e\setminus\e S}\!\overline{C}_{\e T,\e
F_{v}}&\!\!\ra\!\!&\!\!\!\!\Tha^{\e
1}_{S}\big(G,\tto\lbe(\ut)\big)^{\prime}\\
&\ra&\!\!\!\cok j_{\e T,K/F,S}^{\e\prime}\ra 0,
\end{array}
$$
where $\overline{\vartheta}_{_{\be S}}^{\,\prime}$ is the map
\eqref{p}, $\!\Sha^{\e 1}_{S}\big(G,\tto\lbe(\ut)\big)^{\prime}$
(resp., $\!\!\Tha^{\e 1}_{S}\big(G,\tto\lbe(\ut)\big)^{\prime}$) is
the kernel (resp., cokernel) of the natural localization map
\eqref{loc} and the groups $\overline{C}_{\e T,\e F_{v}}$ are given
by \eqref{cbar}.\qed
\end{proposition}

The following immediate corollary of the proposition generalizes
\cite{CAmb}, Theorem 2.4 (note that, if $S\supset B$, then $T_{K}$
has multiplicative reduction over $\ut$ since $T$ has multiplicative
reduction over $U$):

\begin{theorem} Assume that $S\supset B$. Then there exists a canonical exact sequence
$$\begin{array}{rcl}
0\ra\krn j_{\e T,K/F,S}\ra
H^{1}\ng\big(G,\tto\lbe(\ut)\big)^{\e\prime}&\overset{\lambda_{S}}\longrightarrow&
\displaystyle\bigoplus_{v\notin S} H^{1}(G_{w_{v}},\tto(\s
O_{w_{v}}))^{\e\prime}\\\\
&\longrightarrow &\cok j_{\e T,K/F,S}^{\e\prime}\ra 0,
\end{array}
$$
where $\lambda_{S}$ is the canonical localization map
\eqref{loc}.\qed
\end{theorem}

\begin{remark} If $S\supset B$, then Proposition 4.4 shows that
$$
\displaystyle\bigoplus_{\,v\notin S} H^{1}(G_{w_{v}},\tto(\s
O_{w_{v}}))^{\e\prime}=\displaystyle\bigoplus_{v\notin S}\cok\delta_{v}.
$$
Now, for each $v\notin S$, $\cok\delta_{v}$ is canonically
isomorphic to $(\bz/e_{v})^{d_{v}}$, where $d_{v}$ is the dimension
of the largest split subtorus of $T_{F_{v}}$ (see \cite{mrl},
p.1157). Thus the theorem shows that $[\cok j_{\e
T,K/F,S}^{\e\prime}]$ divides $\prod_{\,v\notin S}e_{v}^{\e d_{v}}$.
\end{remark}

When $T$ is an {\it invertible} $F$-torus, $\krn j_{\e T,K/F,S}$
admits a simple description independently of the hypothesis
$S\supset B$ (see Proposition 4.9 below).

\begin{lemma} Let $T$ be an invertible $F$-torus and let $v\notin S$. Then the following hold.
\begin{enumerate}
\item[(a)] $H^{1}(F,T\e)=0$.
\item[(b)] $R^{1}j_{*}T=0$ for the smooth topology on $U$.
\item[(c)] $\Phi_{v}\big(\e\overline{k(v)}\e\big)$ is torsion-free.
\item[(d)] $H^{1}(k(v),\Phi_{v}\be)=0$.
\item[(e)] $\overline{C}_{\le T,\e v}=0$.
\end{enumerate}
\end{lemma}
\begin{proof} Since $T$ is a direct factor of a quasi-trivial $F$-torus, it
suffices to prove the lemma when $T=R_{L/F}(\bg_{m,L})$ for some finite separable extension $L/F$. In this case
(a), (b) and (c) follow from Hilbert's Theorem 90, \cite{B}, Theorem
4.2.2, p.78, and \cite{X}, Lemma 2.6, respectively. Now $T_{F_{v}}=\prod_{\, w\mid v} R_{\e L_{w}/F_{v}}(\bg_{m,\e L_{w}})$, where the
product extends over all primes $w$ of $L$ lying above $v$. Therefore, by \cite{B}, p.34, there exists an isomorphism of
$G_{k(v)}$-modules
$$
\Phi_{v}\big(\e\overline{k(v)}\e\big)=\bigoplus_{w\mid
v}\text{Ind}_{G_{k(v)}}^{G_{k(w)}}\,\bz.
$$
Thus $H^{\e i}\big(k(v),\Phi_{v})=\bigoplus_{\, w\mid v}\e H^{\e
i}\big(k(w),\bz)$ for all $i\geq 0$ and (d) follows. Further, the map
$\vartheta_{v}\colon
T(F_{v})\ra\Phi_{v}(k(v))$ may be identified with the map
$$
\prod_{w\mid v}\text{ord}_{w}\colon\displaystyle\prod_{w\mid v}L_{w}^{*}\ra\displaystyle\prod_{w\mid v}\bz,
$$
which is surjective. Thus $C_{\le T,\e v}=0$, whence (e) holds.
\end{proof}

\begin{proposition} Let $T$ be an invertible $F$-torus and assume that $T_{K}$
has multiplicative reduction over $\ut$. Then there exists a
canonical isomorphism
$$
\krn j_{\e T,K/F,S}=\Sha^{\e 1}_{S}\big(G,\tto\lbe(\ut)\big).
$$
\end{proposition}
\begin{proof} By part (c) of the lemma and Remark 4.3,
$$
\bigoplus_{v\e\in\e B\e\setminus\e
S}H^{\e 1}\!\be\left(I_{w_{v}},T\!\left(K_{w_{v}}
^{\le{\rm{nr}}}\right)\right)^{G_{k(\lbe v)}}=0.
$$
The proposition now follows from Proposition 4.5 using parts (a) and (e) of the lemma.
\end{proof}

We conclude this Section by establishing a generalization of
Dirichlet's Unit Theorem.

Let $V$ be the largest open subscheme of $U$ such that $\sto_{\e
V}:=j^{*}\sto$ is a torus, where $j\colon V\ra U$ is the canonical
inclusion. Then there exists a canonical exact sequence of smooth
sheaves on $U$
$$
0\ra j_{\e !}\sto_{\e V}\ra\sto\ra\displaystyle\bigoplus_{v\in
B\e\setminus\le S}(i_{v})_{*}\sto_{\e v}\ra 0,
$$
where $j_{\e !}$ is the extension-by-zero functor and, for each
$v\in B\setminus S$, $\sto_{\e v}:=i_{v}^{*}\sto$. See \cite{T}, Proposition 8.2.1, p.142.
Consequently, there exists a canonical exact sequence of abelian
groups
\begin{equation}\label{dut}
0\ra\sto_{\e V}(V)\ra\sto(U)\ra\displaystyle\bigoplus_{v\in
B\e\setminus\le S}\!\sto_{\e v}(k(v)).
\end{equation}
Each $\sto_{\e v}$ is an affine, {\it connected}, smooth group
scheme over the finite field $k(v)$. Thus $\sto_{\e v}(k(v))$ is
finite for every $v\in B\setminus S$ and \eqref{dut} shows that
$\sto_{\e V}(V)$ and $\sto(U)$ have the same $\bz$-rank. The
following result generalizes Dirichlet's Unit Theorem.
\begin{theorem} We have
$$
{\rm{rank}}_{\e\bz}(\sto(U))={\rm{rank}}_{\e\bz}(X^{G_{\lbe
F}})\big(\#( S\cup B\e)-1\big).
$$
\end{theorem}
\begin{proof} As noted above, $\sto_{\e
V}(V)$ and $\sto(U)$ have the same $\bz$-rank. By \cite{mrl}, proof
of Corollary 3.8,
$$
\sto_{\e V}(V)=\text{Hom}\left(X^{G_{\be F}}\be,{\s O}_{ F,\e
S\e\cup\e B}^{\e *}\e\right).
$$
The result is now immediate from the classical Dirichlet Unit
Theorem.
\end{proof}

\begin{remark} Let $F$ be a number field. In \cite{Sh1}, a
generalization of Dirichlet's Unit Theorem was obtained for any
$F$-torus $T$ using adelic methods and (implicitly) the standard
model of $T$. The result is comparable to the above, but the
proof is significantly more involved.
\end{remark}

\section{The capitulation cokernel}

In this Section we prove Theorem 1.2 (this is Theorem 5.4 below).

We assume that $K$ splits $T$. In particular, the action of $G_{\lbe
F}$ on $X$ factors through $G$ and $X^{G_{\lbe F}}=X^{G}$.

\smallskip

For any prime $v$ of $F$, $T_{K_{w_{v}}}$ is a split torus. Thus there exists an exact sequence
$$
0\ra\tto\lbe(\lbe\s O_{w_{v}}\lbe)\ra T(K_{w_{v}})\ra\Phi_{w_{v}}\be(k(w_{v}))\ra 0,
$$
where $\Phi_{w_{v}}\be(k(w_{v}))$ is a free abelian group with
trivial $G_{w_{v}}$-action. It follows that there exists an exact
sequence
\begin{equation}\label{pi2}
0\ra H^{2}(G_{w_{v}},\tto\lbe(\lbe\s O_{w_{v}}\lbe))\ra  H^{2}(G_{w_{v}},T(K_{w_{v}}))\overset{\pi_{v}}
\longrightarrow H^{2}(G_{w_{v}},\Phi_{w_{v}}\be(k(w_{v}))),
\end{equation}
where $\pi_{v}$ is induced by the projection
$T(K_{w_{v}})\ra\Phi_{w_{v}}\be(k(w_{v}))$. Further, the map
$\partial_{\e 2}\colon C_{\e T,K,S_{K}}^{\e G}\ra H^{\e
1}\big(G,T(K)/\,\tto\lbe\big(\ut\big)\big)$ given by \eqref{p2} is
surjective. Let
\begin{equation}\label{pi}
\pi\colon H^{1}(G,T(K))\ra H^{\e
1}\big(G,T(K)/\,\tto\lbe\big(\ut\big)\big)
\end{equation}
be induced by the projection map $T(K)\ra
T(K)/\,\tto\lbe\big(\ut\big)$. Then there exists a canonical exact
sequence
$$
0\ra\partial_{\e 2}^{\e -1}(\img\pi)\ra C_{\e T,\e K,\e S_{K}}^{\e
G} \overset{\partial_{\e 2}}\longrightarrow\cok\pi\ra 0.
$$
Let $\widetilde{\jmath}_{_{T,\e K/F,S}}^{\,\,\prime}$ be the composite
\begin{equation}\label{jtilda}
C_{\e T,F,S}\overset{j_{_{T,\e K/F,S}}^{\e\prime}}\longrightarrow(C_{\e T,K,S_{K}})^{\e G}_{\text{trans}}\hookrightarrow\partial_{\e 2}^{\e -1}(\img \pi),
\end{equation}
where $j_{_{T,\e K/F,S}}^{\e\prime}$ is the map \eqref{j'} and the
second map is the canonical inclusion of $(C_{\e T,K,S_{K}})^{\e
G}_{\text{trans}}=\krn\,\partial_{\e 2}$ into $\partial_{\e 2}^{\e
-1}(\img \pi)$. Then the following holds (cf. \cite{CAmb},
Proposition 2.2).

\begin{lemma} Assume that $K$ splits $T$. Then there exists a canonical exact sequence
$$
0\ra\cok \widetilde{\jmath}_{_{T,\e K/F,S}}^{\,\,\prime}\ra
\cok j_{_{T,\e K/F,S}}\ra \cok\pi\ra 0,
$$
where $\widetilde{\jmath}_{_{T,\e K/F,S}}^{\,\,\prime}$ and $\pi$
are the maps \eqref{jtilda} and \eqref{pi}, respectively.\qed
\end{lemma}

Now define
\begin{equation}\label{bs}
B_{S}(G,T)=\krn\!\left[H^{2}(G,T(K))\ra\displaystyle\bigoplus_{v\notin S}H^{\e
2}(G_{w_{v}},\Phi_{w_{v}}\be(k(w_{v})))\right],
\end{equation}
where the map involved is induced by \eqref{theta}. Then the
following holds (cf. \cite{CAmb}, Proposition 3.1)

\begin{proposition} Assume that $K$ splits $T$. Then there exists a canonical exact sequence
$$
\begin{array}{rcl}
0&\ra&\cok\widetilde{\jmath}_{_{T,\e K/F,S}}^{\,\,\prime}\ra
\cok j_{_{T,\e K/F,S}}\ra H^{2}\ng\big(G,\tto\lbe(\ut)\big)\ra B_{S}(G,T)\\\\
&\ra& H^{1}\big(G, C_{\e T,K,S_{K}}\big)\ra
H^{3}\ng\big(G,\tto\lbe(\ut)\big),
\end{array}
$$
where $\widetilde{\jmath}_{_{T,\e K/F,S}}^{\,\,\prime}$ is the map
\eqref{jtilda} and $B_{S}(G,T)$ is the group \eqref{bs}.
\end{proposition}
\begin{proof} The exact sequence \eqref{seq1} induces an exact sequence
$$\begin{array}{rcl}
0&\ra&\cok\pi\ra H^{2}\ng\big(G,\tto\lbe(\ut)\big)\ra H^{2}(G,T(K))
\\\\
&\ra& H^{\e 2}\big(G,T(K)/\,\tto\lbe\big(\ut\big)\big)\ra
H^{3}\ng\big(G,\tto\lbe(\ut)\big)
\end{array}
$$
which generalizes \cite{I}, p.189, line -1. Now essentially the same
argument given in \cite{I}, pp.189-191, to derive the exact sequence
\cite{I}, p.191, line 5, yields the exact sequence
$$
\begin{array}{rcl}
0&\ra&\cok\pi\ra H^{2}\ng\big(G,\tto\lbe(\ut)\big)\ra B_{S}(G,T)\ra H^{1}\big(G, C_{\e T,K,S_{K}}\big)\\\\
&\ra& H^{3}\ng\big(G,\tto\lbe(\ut)\big).
\end{array}
$$
The proposition now follows from the previous lemma.
\end{proof}

Since $\!\!\Sha^{2}(G,T(K))=\!\!\Sha^{2}(T)$ by \cite{San}, Lemma
$1.9\e$\footnote{The proof of this lemma is valid in the function
field case as well.}, the definition of $B_{S}(G,T)$ shows that
there exists an exact sequence
\begin{equation}\label{Shabs}
\begin{array}{rcl}
0&\ra&\Sha^{2}(F,T)\ra B_{S}(G,T)\ra H^{2}(G,T(K))/\!\be\Sha^{2}(F,T)\\\\
&\ra&\displaystyle\bigoplus_{v\notin S}H^{\e
2}(G_{w_{v}},\Phi_{w_{v}}\be(k(w_{v}))).
\end{array}
\end{equation}
On the other hand, there exists a canonical exact commutative diagram
\begin{equation}\label{d2}
\xymatrix{H^{2}(F,T)/\!\be\Sha^{2}(F,T)\,\,\ar@{^{(}->}[r]\ar[d]
& \displaystyle\bigoplus_{\text{all $v$}}H^{2}(F_{v},T)\ar@{->>}[r]\ar[d]&(X^{G})^{D}\ar[d]\\
H^{2}(K,T)^{G}\,\,\ar@{^{(}->}[r]& \displaystyle\bigoplus_{\text{all $v$}}H^{2}(K_{w_{v}},T)^{G_{w_{v}}}\ar[r]&(X_{G})^{D}.}
\end{equation}
For the exactness of the top row, see \cite{Oe}, Theorem 2.7(b),
p.52. The bottom row is the beginning of the long $G$-cohomology
sequence associated to the exact sequence
$$
0\ra H^{2}(K,T)\ra\displaystyle\bigoplus_{\,\text{all $w$}}H^{2}(K_{w},T)\ra X^{D}\ra 0
$$
(the latter sequence is the direct sum of $\dimn\e T$ copies of the
well-known exact sequence $0\ra\br(K)\ra \bigoplus_{\,\text{all
$w$}}\br(K_{w}) \ra \bq/\bz\ra 0$). The first two vertical maps in
\eqref{d2} are induced by the restriction map. Their kernels are
$H^{2}(G,T(K))/\!\be\Sha^{2}(F,T)\e$ and $\e\bigoplus_{\,\text{all
$v$}}\be H^{2}(G_{w_{v}},T(K_{w_{v}}))$, respectively (see \cite{S},
Proposition 5, p.117). The right-hand vertical map in \eqref{d2}  is
the dual of the norm map $X_{G}\ra X^{G}$, whose cokernel is
$\widehat{H}^{\e 0}(G,X)$ (by definition). Thus \eqref{d2} induces
an exact sequence
$$
0\ra H^{2}(G,T(K))/\!\be\Sha^{2}(F,T)\ra\displaystyle\bigoplus_{\,\text{all $v$}}H^{2}(G_{w_{v}},T(K_{w_{v}}))\ra \widehat{H}^{\e 0}(G,X)^{D}.
$$
The preceding exact sequence is the top row of a commutative diagram
\begin{equation}\label{diag}
\xymatrix{H^{2}(G,T(K))/\!\be\Sha^{2}(F,T)\ar@{^{(}->}[r]\ar[dr]&\displaystyle\bigoplus_{\text{all $v$}}H^{2}(G_{w_{v}},T(K_{w_{v}}))
\ar[d]^{\pi}\ar[r]&\widehat{H}^{\e 0}(G,X)^{D}\\
&\displaystyle\bigoplus_{v\notin S}H^{2}(G_{w_{v}},\Phi_{w_{v}}\be(k(w_{v}))),&}
\end{equation}
where the map $\pi$ has $v$-components $\pi_{v}$ if $v\notin S$ (see
\eqref{pi2}) and 0 otherwise. For any $v$, set
\begin{equation}\label{h2w}
H^{2}(G_{w_{v}},T(K_{w_{v}}))^{\e\prime}=\begin{cases}
H^{2}(G_{w_{v}},T(K_{w_{v}}))\kern 1.35em\text{if $v\in S$}\\
H^{2}(G_{w_{v}},\tto\lbe(\lbe\s O_{w_{v}}\lbe))\quad\text{if  $v\notin S$}\end{cases}.
\end{equation}
Further, let $R$ denote the set of non-archimedean primes of $F$
which ramify in $K$. Then $H^{2}(G_{w_{v}},\tto\lbe(\lbe\s
O_{w_{v}}\lbe))=0$ if $v\notin R$ (this is a well-known fact for the
split $K$-torus $T_{K}$. See, e.g., \cite{S}). We conclude that
$$
\krn\,\pi=\displaystyle\bigoplus_{v\e\in\e S\e\cup\le R}H^{2}(G_{w_{v}},T(K_{w_{v}}))^{\e\prime}.
$$
On the other hand, it is clear that the kernel of the oblique map in
\eqref{diag} is the same as the kernel of the map $\krn\,\pi\ra
\widehat{H}^{\e 0}(G,X)^{D}$ induced by the right-hand horizontal
map in \eqref{diag}. Combining this information with \eqref{Shabs},
we obtain the following generalization of \cite{CAmb}, Lemma 3.2.

\begin{lemma} Assume that $K$ splits $T$. Then there exists a canonical exact sequence
$$
0\ra\Sha^{2}(T)\ra B_{S}(G,T)\ra\displaystyle\bigoplus_{v\e\in\e
S\e\cup\le R}H^{2}(G_{w_{v}},T(K_{w_{v}}))^{\e\prime}
\ra\widehat{H}^{\e 0}(G,X)^{D},
$$
where $B_{S}(G,T)$ is the group \eqref{bs} and the groups
$H^{2}(G_{w_{v}},T(K_{w_{v}}))^{\e\prime}$ are given by \eqref{h2w}.
\end{lemma}

We now combine Proposition 5.2 and Lemma 5.3 to obtain the following
generalization of \cite{CAmb}, Theorem 3.3.

\begin{theorem} Assume that $K$ splits $T$ and let $R$ be the set of primes
of $F$ which ramify in $K$. Then there exists a canonical exact sequence
$$
\begin{array}{rcl}
0&\ra&\cok\widetilde{\jmath}_{_{T,\e K/F,S}}^{\,\,\prime}\ra
\cok j_{_{T,\e K/F,S}}\ra H^{2}\ng\big(G,\tto\lbe(\ut)\big)\ra B_{S}(G,T)\\\\
&\ra& H^{1}\big(G, C_{\e T,K,S_{K}}\big)\ra
H^{3}\ng\big(G,\tto\lbe(\ut)\big),
\end{array}
$$
where $\widetilde{\jmath}_{_{T,\e K/F,S}}^{\,\,\prime}$ is the map
\eqref{jtilda} and the group $B_{S}(G,T)$ fits into an exact
sequence
$$
0\ra\Sha^{2}(T)\ra B_{S}(G,T)\ra\displaystyle\bigoplus_{v\e\in\e
S\e\cup\le R}H^{2}(G_{w_{v}},T(K_{w_{v}}))^{\e\prime}
\ra\widehat{H}^{\e 0}(G,X)^{D}.
$$
Here, the groups $H^{2}(G_{w_{v}},T(K_{w_{v}}))^{\e\prime}$ are
given by \eqref{h2w}.\qed
\end{theorem}

\section{Invertible resolutions and class groups}
The following result was stated as Theorem 1.3 in the Introduction.

\begin{theorem} Let $T$ be an $F$-torus which admits an invertible resolution $0\ra T_{1}\ra Q\ra
T\ra 0$, where $T_{1}$ is invertible and $Q$ is quasi-trivial. Let $\s T_{1}, \s Q$ and $\s T$ be the
N\'eron-Raynaud models of $T_{1}, Q$ and $T$ over $U$, respectively. Then there
exists a canonical exact sequence of abelian groups
$$
0\ra\sto_{1}(U)\ra\sqo(U)\ra\sto(U)\ra C_{T_{1},\le F,S}\ra
C_{Q,\e F,S}\ra C_{T,\e F,S}\ra 0.
$$
\end{theorem}
\begin{proof} Let $v\notin S$. By Lemma 4.8, parts (b) and (c), and
\cite{B}, Theorem 5.3.1, p.99, there exists an exact sequence
$$
0\ra\Phi_{v}(T_{1})\big(\e\overline{k(v)}\e\big)\ra\Phi_{v}(Q)\big(\e\overline{k(v)}\e\big)\ra\Phi_{v}(T)\big(\e\overline{k(v)}\e\big)\ra 0.
$$
By Lemma 4.8(d), the preceding exact sequence induces an exact sequence
$$
0\ra\Phi_{v}(T_{1})(k(v))\ra\Phi_{v}(Q)(k(v))\ra\Phi_{v}(T)(k(v))\ra
0.
$$
On the other hand, by Lemma 4.8(a), there exists an exact
sequence
$$
0\ra T_{1}(F)\ra Q(F)\ra T(F)\ra 0.
$$
Thus there exists a canonical exact commutative diagram
$$
\xymatrix{0\ar[r] & T_{1}(F)\ar[d]\ar[r]
& Q(F)\ar@{->>}[r]\ar[d] & T(F)\ar[d]\\
0\ar[r]&\displaystyle\bigoplus_{v\notin
S}\Phi_{v}(T_{1})(k(v))\ar[r]&\displaystyle\bigoplus_{v\notin
S}\Phi_{v}(Q)(k(v))\ar@{->>}[r]&\displaystyle\bigoplus_{v\notin
S}\Phi_{v}(T)(k(v)),}
$$
where the vertical maps are the maps \eqref{theta} for the $F$-tori
$T_{1}, Q$ and $T$. The theorem now follows by applying the
snake lemma to the above diagram using the analog of Proposition 4.1
over $F$.
\end{proof}

\begin{remark} Since both $T_{1}$ and $Q$ are direct factors of quasi-trivial tori,
their $S$-class groups $C_{T_{1},\le F,S}$ and $C_{Q,\e F,S}$ can be
described in terms of $S$-ideal class groups of global fields. Thus
the theorem may be interpreted as saying that $C_{\e T,\e F,S}$ can
be ``resolved" in terms of classical objects.
\end{remark}

We now develop some applications of the theorem.

\smallskip

We begin by examining the behavior of the maps
$\vartheta_{v}$ introduced in Section 4 under the Weil restriction functor (see \cite{BLR}, \S7.6, for basic information on this functor).
Let $T$ be any $F$-torus and let $T^{\e\prime}=R_{\e K/F}(T_{\lbe K})$, where $K/F$ is
a finite separable extension. For any prime $v\notin S$,
$T^{\e\prime}_{F_{v}}=\prod_{\, w\mid v} R_{\e K_{w}/F_{v}}(T_{\lbe
K_{w}})$, where the product extends over all primes $w$ of $K$ lying
above $v$. Therefore, by \cite{B}, p.34, there exists an isomorphism of
$G_{k(v)}$-modules
$$
\Phi_{v}\big(\e T^{\e\prime}\e\big)\big(\e\overline{k(v)}\e\big)=\bigoplus_{w\mid
v}\text{Ind}_{G_{k(v)}}^{G_{k(w)}}\Phi_{w}(T_{\lbe
K})\big(\e\overline{k(v)} \e\big).
$$
Thus $H^{\e i}\big(k(v),\Phi_{v}\big(\e T^{\e\prime}\e\big)\big)=\bigoplus_{\, w\mid v}\e H^{\e
i}\big(k(w),\Phi_{w}(T_{\lbe K})\big)$ for all $i\geq 0$. In
particular, $\Phi_{v}(T^{\e\prime})(k(v))=\bigoplus_{\e w\mid
v}\Phi_{w}(T_{\lbe K})(k(w))$ and it follows that the map
$\vartheta_{\lbe v}^{\e\prime}\colon
T^{\e\prime}(F_{v})\ra\Phi_{v}(T^{\e\prime})(k(v))$ may be identified with the map
$$
\bigoplus_{w\mid v}\vartheta_{w}\colon \displaystyle\prod_{w\mid v}T(K_{w})\ra \displaystyle\bigoplus_{w\mid v}\Phi_{w}(T_{\lbe K})(k(w)).
$$
Now there exists a canonical commutative diagram
$$
\xymatrix{T(F_{v})\ar@{^{(}->}[d]\ar[r]^(.45){\vartheta_{v}}
&\Phi_{v}(T)(k(v))\ar[d]^(.45){\mu_{\le v}}\\
\displaystyle\prod_{w\mid v}T(K_{w})\ar[r]^(.4){\oplus\le\vartheta_{\lbe w}}&\displaystyle\bigoplus_{w\mid v}\Phi_{w}(T_{K})(k(w)),}
$$
where the vertical maps are induced by the canonical embedding $T\hookrightarrow T^{\e\prime}$ and the identification
$\Phi_{v}(T^{\e\prime})(k(v))=\bigoplus_{\e w\mid
v}\Phi_{w}(T_{\lbe K})(k(w))$. Thus $\vartheta_{\lbe v}^{\e\prime}=\oplus_{\,w\mid v}\vartheta_{w}$ induces a map
$$
\overline{\vartheta}_{\lbe v}^{\e\prime}\colon \prod_{w\mid v}T(K_{w})\big/T(F_{v})\ra\cok\mu_{\le v}.
$$
Let $\varphi_{\e T,\e S}$ be the composite
\begin{equation}\label{phi T}
T(K)/T(F)\ra\prod_{v\notin S}\left[\textstyle\prod_{\, w\mid v}\lbe T(K_{w})/T(F_{v})\right]\overset{\oplus\e\overline{\vartheta}_{\le v}^{\e\prime}}\longrightarrow\displaystyle\bigoplus_{v\notin S}\cok\mu_{\le v},
\end{equation}
where the first map is the canonical $S$-localization map.

When $T=\bg_{m,F}$, the map $\mu_{\le v}$ agrees with the injection $\bz\ra\oplus_{\e w\mid
v}\,\bz,\,m\mapsto(e_{w}m)_{w\mid v}$ (this follows from the formula ``$\,\text{ord}_{\e
w}\be(x)=e_{w}\e\text{ord}_{\e v}(x)$ for $x\in F_{v}^{\e *}\,$'' already cited) and $\varphi_{\e S}:=\varphi_{\e \bg_{m},\e S}$ is the map
\begin{equation}\label{phi}
K^{*}/F^{*}\ra \bigoplus_{v\notin S}\,(\oplus_{\e w\mid v}\e\bz)/\e\bz
\end{equation}
induced by the maps $K^{*}\ra\oplus_{\e w\mid v}\e\bz, x\mapsto
(\text{ord}_{w}(x))_{w\mid v}$, for $v\notin S$.

Now let $T$ be any invertible $F$-torus and let $K$ be a finite
Galois extension of $F$ (with Galois group $G$) such that $T_{K}$ is
quasi-trivial ($T$ now plays the role of the torus called $T_{1}$ in
the statement of Theorem 6.1). Then the quotient $F$-torus
$P:=R_{K/F}(T_{K})/\e T$ admits the canonical invertible resolution
$$
0\ra T\ra R_{K/F}(T_{K})\ra P\ra 0.
$$
Let $\widetilde{\s T}$ be the N\'eron-Raynaud model of $T_{K}$ over
$\ut$. Then the identity component of the N\'eron-Raynaud model of
$R_{K/F}(T_{K})$ over $U$ is $R_{\e\ut/U}(\tto)$ (see Section 2).
Further, by Remark 2.2, the N\'eron-Raynaud $S$-class group of
$R_{K/F}(T_{K})$ is canonically isomorphic to $C_{\e T,\e K,\e
S_{K}}$. Thus the following is an immediate consequence of Theorem 6.1.

\begin{corollary} Let $T$ be an invertible $F$-torus, let $K$ be
a finite Galois extension of $F$ such that $T_{K}$ is quasi-trivial
and let $P=R_{K/F}(T_{K})/\e T$. Write $\s T$ and $\s P$,
respectively, for the N\'eron-Raynaud models of $\e T$ and $P$ over
$U$ and let $\widetilde{\s T}$ be the N\'eron-Raynaud model of
$T_{K}$ over $\ut$. Then there exists a canonical exact sequence
$$
0\ra\sto(U)\ra\tto\be\big(\e\ut\e\big)\ra\spo(U)\ra C_{\e T,\le
F,S}\ra C_{\e T,\e K,\e S_{K}}\ra C_{\e P,\e F,\e S}\ra 0.\qed
$$
\end{corollary}

\smallskip

The map $C_{\e T,\le F,S}\ra C_{\e T,\e K,\e S_{K}}$ appearing in
the exact sequence of the corollary is the composite of the
capitulation map \eqref{cap} and the canonical inclusion $C_{\e T,\e
K,\e S_{K}}^{\e G}\hookrightarrow C_{\e T,\e K,\e S_{K}}$ (see
Section 2). On the other hand, by Proposition 4.1, $\spo(U)$ is the
kernel of the map $P(F)\ra\bigoplus_{\,v\notin S}\Phi_{v}(P)(k(v))$,
and the proof of Theorem 6.1 shows that the latter map can be
identified with the map \eqref{phi T}. Thus the following holds.

\begin{corollary} Let $T$ be an invertible $F$-torus, let $K$ be
a finite Galois extension of $F$ such that $T_{\be K}$ is
quasi-trivial and let $P=R_{K/F}(T_{\be K})/\e T$. Let $ j_{\e T,\e
K/F,\e S}\colon C_{\e T,\le F,S}\ra C_{\e T,\e K,\e S_{K}}^{\e G}$
be the $S$-capitulation map for $T$. Then there exist canonical
exact sequences
$$
0\ra{\rm{Coker}}\, j_{\e T,\e K/F,\e S}\ra C_{\e P,\e F,S}\ra C_{\e T,\e K,\e S_{K}}/C_{\e T,\e K,\e S_{K}}^{\e G}\ra 0
$$
and
$$
0\ra\tto\be\big(\e\ut\e\big)/\e\sto(U)\ra\krn\varphi_{\, T,\e
S}\ra{\rm{Ker}}\,j_{\e T,\e K/F,\e S}\ra 0,
$$
where $\varphi_{\, T,\e S}$ is the map \eqref{phi T}.
\end{corollary}

\begin{remark} S.-I.Katayama \cite{K}, Theorem 2, and V.Voskresenskii \cite{V}, Chapter 7, \S20, obtained results
for the class group of the standard model of
$P=R_{K/F}(\bg_{m,K})/\e \bg_{m,F}$ which bear some resemblance to
the case $T=\bg_{m,F}$ of Corollary 6.4.
\end{remark}

Specializing Corollary 6.4 to $T=\bg_{m,\e F}$, we obtain the
following result on the classical $S$-capitulation map $j_{K/F,\e
S}$ which supplements those obtained in \cite{CAmb}.

\begin{corollary} Let $K/F$ be a finite Galois extension of global fields with Galois group $G$ and let $P=R_{K/F}(\bg_{m,K})/\e \bg_{m,F}$ be the projective group
torus. Then there exist canonical exact sequences
$$
0\ra{\rm{Coker}}\, j_{K/F,\e S}\ra C_{\e P,\e F,S}\ra C_{K,\e S_{K}}/C_{K,\e S_{K}}^{\e G}\ra 0
$$
and
$$
0\ra\s O_{K,\e S_{K}}^{*}/\e\s O_{F,\e S}^{*}\ra\krn\!\!\left[\e
K^{*}/F^{*} \overset{\varphi_{_{\be
S}}}\longrightarrow\bigoplus_{v\notin S}\, (\oplus_{\e w\mid
v}\e\bz)/\e\bz\e\right]\ra\krn j_{\e K/F,\e S}\ra 0,
$$
where $\varphi_{S}$ is the map \eqref{phi}.\qed
\end{corollary}

Now let $T$ be an arbitrary $F$-torus. The {\it Ono invariants} of
$T$ are defined as follows. Choose a flasque resolution
$$
0\ra T_{1}\ra Q_{1}\ra T\ra 0
$$
and set
$$
E_{f,\e S}(T)=h_{\e Q_{1},\e F,\e S}/\e h_{\e T_{1},\e F,\e S}\e h_{\e
T,\e F,\e S}.
$$
Then $E_{f,\e S}(T)$ is in fact independent of the flasque
resolution of $T$ used to define it and depends only on $S$ and $T$
(see \cite{V}, p.50, lines 18-28). Similarly, if $0\ra T\ra Q_{2}\ra
T_{2}\ra 0$ is a coflasque resolution of $T$, then
$$
E_{\e c,\e S}(T)=h_{\e Q_{2},\e F,\e S}/\e h_{\e T_{2},\e F,\e S}\e
h_{\e T,\e F,\e S}
$$
is an invariant of the pair $S,T\e$\footnote{\e Even though the
duality $T\mapsto T^{\vee}$ transforms flasque resolutions of $T$
into coflasque resolutions of the dual torus $T^{\vee}$, and
viceversa, it is unlikely that the equality $E_{f,\e S}(T)=E_{c,\e
S}\big(T^{\vee}\big)$ will hold for an arbitrary torus $T$.}. The
latter invariant was studied in \cite{K,M,O,Sa}, for $T$ equal
either to a norm or a binorm torus, using certain integral models of
$T$ and the canonical coflasque resolutions that define these tori.
The invariant $E_{f,\e S}(T)$ was computed in \cite{K}, \S4, for
$T=R_{K/F}(\bg_{m,K})/\bg_{m,F}$, where $F$ is a number field, using
the standard model of $T$ and the canonical flasque resolution that
defines this torus.

\begin{theorem} Let $T$ be an invertible $F$-torus and let $K$ be
a finite Galois extension of $F$, with Galois group $G$, such that
$T_{K}$ is quasi-trivial and has multplicative reduction over $\ut$.
Let $P=R_{K/F}(T_{K})/\e T$. Then
$$
E_{f,\e S}(P)=\left[\!\!\Sha^{1}_{S}\big(G,\tto\lbe\big(\ut\big)\big)\right]^{-1}.
$$
\end{theorem}
\begin{proof} Corollary 6.3 and the comment that follows it show that
$E_{f,\e S}(P)=[\krn j_{\e T,\e K/F,\e S}]^{-1}$, where $j_{\e T,\e K/F,\e S}$ is the $S$-capitulation map for
the torus $T$. The theorem now follows from Proposition 4.9.
\end{proof}

\section{Flasque resolutions and class groups}

As noted in the Introduction, every $F$-torus $T$ admits a flasque
resolution
\begin{equation}\label{f1}
0\ra T_{1}\ra Q\ra T\ra 0,
\end{equation}
where $Q$ is quasi-trivial and $T_{1}$ is flasque. Since, in
general, the induced sequence of $F$-rational points $0\ra
T_{1}(F)\ra Q(F)\ra T(F)\ra 0$ will not be exact and the sheaf
$R^{1}j_{*}T_{1}$ will not be trivial for the smooth topology on
$U$, it is unreasonable to expect that \eqref{f1} will induce simple
resolutions of $C_{\e T,\e F,\e S}$ such as that of Theorem 6.1 (see
the proof of that theorem). However, we will see below that
\eqref{f1} does induce (at least under a certain simplifying
assumption) a resolution of $C_{\e T,\e F,\e S}$ which is
interesting in spite of its increased complexity.

\smallskip

We begin by reviewing the concept of $R$-equivalence on tori. Let
$T$ be an $F$-torus. Two points $x,y\in T(F)$ are said to be {\it
$R$-equivalent}, written $x\sim y$, if there exists an $F$-rational
map $f\colon \ba _{\lbe F}^{\be 1}\dashrightarrow T$, defined at $0$
and $1$, such that $f(0)=x$ and $f(1)=y\e$\footnote{\,This is in
fact the definition of {\it strict} (or direct) $R$-equivalence. We
are using the fact that, on a torus, $R$-equivalence and strict
$R$-equivalence coincide \cite{CTS}, Theorem 2, p.199.}. If $A$ is
any subgroup of $T(F)$, then the subset of $A$ of all points which
are $R$-equivalent to $1$ is in fact a subgroup and will be denoted
by $R\e A$. We will write $A/R$ for $A/R\e A$. Then any flasque
resolution \eqref{f1} induces a short exact sequence
\begin{equation}\label{rat points}
0\ra T_{1}(F)\ra Q(F)\ra R\, T(F)\ra 0
\end{equation}
and an isomorphism
$$
T(F)/R=H^{1}(F,T_{1}).
$$
See \cite{CTS}, Theorem 2, p.199, or \cite{V}, \S17.1. Now recall
the map $\vartheta_{_{\! S}}\colon T(F)\ra\bigoplus_{v\notin
S}\Phi_{v}(k(v))$ defined at the beginning of Section 4. It induces
a map
\begin{equation}\label{theta R}
\vartheta_{_{\! S}}^{\le R}\colon R\, T(F)\ra\bigoplus_{v\notin
S}\Phi_{v}(k(v))
\end{equation}
whose cokernel will be denoted by $C_{\e T,\e F,\e S}^{\e R}$ and
called the {\it $R$-equivalence class group} of $T$. Thus $C_{\e
T,\e F,\e S}^{\e R}$ is defined by the exactness of the sequence
\begin{equation}\label{R seq}
0\ra R\,\sto(U)\ra R\, T(F)\overset{\vartheta_{_{\!\be S}}^{\lbe R}}
\longrightarrow\displaystyle\bigoplus_{v\notin S}\Phi_{v}(k(v)) \ra
C_{\e T,\e F,\e S}^{\e R}\ra 0.
\end{equation}
Clearly, if $R$-equivalence is {\it trivial}\, on $T(F)$ (as is the
case, for example, if $T$ admits an invertible resolution by Lemma
4.8(a)), then $C_{\e T,\e F,\e S}^{\e R}$ and $C_{\e T,\e F,\e S}$
coincide. In general, the following holds.

\begin{lemma} There exists an exact sequence of finite abelian groups
$$
0\ra\sto(U)/R\ra T(F)/R\ra C_{\e T,\e F,\e S}^{\e R}\ra C_{\e T,\e F,\e S}\ra 0,
$$
where $ C_{\e T,\e F,\e S}^{\e R}$ is defined by \eqref{R seq}.
\end{lemma}
\begin{proof} This follows at once from the exact commutative diagram
$$
\xymatrix{0\ar[r] & R\e T(F)\ar[d]\ar[r]
& T(F)\ar[r]\ar[d] & T(F)/R\ar[d]\ar[r]&0\\
0\ar[r]&\displaystyle\bigoplus_{v\notin
S}\Phi_{v}(k(v))\ar@{=}[r]&\displaystyle\bigoplus_{v\notin
S}\Phi_{v}(k(v))\ar[r]&0\ar[r]&0}
$$
and \cite{CTS}, Corollary 2, p.200 (for the finiteness statement).
\end{proof}

\begin{remark} If $T$ is split by a finite Galois extension of degree $n$ and
exponent $e$, then $T(F)/R$ is annihilated by $n/e$ \cite{CTS},
Corollary 4, p.200. Thus $C_{\e T,\e F,\e S}^{\e R}$ and $C_{\e T,\e
F,\e S}$ have the same $p$-primary components for every prime $p$
which does not divide $n/e$.
\end{remark}

Now, for each $v\notin S$, the exact sequence
$$
0\ra\Phi_{v}\big(\e\overline{k(v)}\e\big)_{\text{tors}}\ra
\Phi_{v}\big(\e\overline{k(v)}\e\big)\ra
\Phi_{v}\big(\e\overline{k(v)}\e\big)/\text{tors}\ra 0
$$
and the isomorphism of $G_{k(v)}$-modules
$\Phi_{v}\big(\e\overline{k(v)}\e\big)_{\text{tors}}=H^{1}(F_{\lbe
v}^{\le\rm{nr}},T)$ \cite{mrl}, proof of Lemma 3.3, induce an exact
sequence
$$\begin{array}{rcl}
0\ra H^{1}(F_{\lbe
v}^{\le\rm{nr}},T)^{G_{k(v)}}\ra\Phi_{v}(k(v))&\overset{\eta_{_v}}\longrightarrow&
\big(\Phi_{v}\big(\e\overline{k(v)}\e\big)/\text{tors}\big)^{G_{k(v)}}\\\\
&\longrightarrow & H^{1}\big(k(v),H^{1}(F_{\lbe
v}^{\le\rm{nr}},T)\big).
\end{array}
$$
Clearly, if $T$ splits over $F_{\lbe v}^{\le\rm{nr}}$, i.e., has
multiplicative reduction at $v$, then $H^{1}(F_{\lbe
v}^{\le\rm{nr}},T)=0$ by Hilbert's Theorem 90. Consequently, the
preceding exact sequences induce an exact sequence
\begin{equation}\label{7.5}
\begin{array}{rcl} \displaystyle\bigoplus_{v\e\in\e
B\e\setminus S}H^{1}(F_{\lbe
v}^{\le\rm{nr}},T)^{G_{k(v)}}&\hookrightarrow&\displaystyle\bigoplus_{v\notin
S}\Phi_{v}(k(v))\overset{\eta_{_S}}\longrightarrow
\displaystyle\bigoplus_{v\notin
S}\big(\Phi_{v}\big(\e\overline{k(v)}\e\big)/\text{tors}\big)^{G_{k(v)}}\\\\
&\ra& \displaystyle\bigoplus_{v\e\in\e B\e\setminus
S}H^{1}\big(k(v),H^{1}(F_{\lbe v}^{\le\rm{nr}},T)\big),
\end{array}
\end{equation}
where $\eta_{_S}=\bigoplus_{\,v\notin S}\eta_{_v}$ and $B$ is the
set of primes of primes of $F$ where $T$ has bad reduction. Now let
$\vartheta_{_{\!\be S}}^{\e *}$ denote the composite
\begin{equation}\label{theta*}
T(F)\overset{\vartheta_{_{\be
S}}}\longrightarrow\displaystyle\bigoplus_{v\notin S}
\Phi_{v}(k(v))\overset{\eta_{_S}}\longrightarrow
\displaystyle\bigoplus_{v\notin
S}\big(\Phi_{v}\big(\e\overline{k(v)}\e\big)/\text{tors}\big)^{G_{k(v)}}
\end{equation}
and set $C_{\e T,\e F,\e S}^{\e *}=\cok\vartheta_{_{\be S}}^{*}$.
Thus $C_{\e T,\e F,\e S}^{\e *}$ is defined by the exactness of the
sequence
\begin{equation}\label{seq star}
0\ra \sto(U)\ra T(F)\overset{\vartheta_{_{\!\! S}}^{\e *}}
\longrightarrow\displaystyle\bigoplus_{v\notin
S}\big(\Phi_{v}\big(\e\overline{k(v)}\e\big)/\text{tors}\big)^{G_{k(v)}}
\ra C_{\e T,\e F,\e S}^{\e *}\ra 0.
\end{equation}

\begin{lemma} There exists a canonical exact sequence
$$\begin{array}{rcl}
0\ra\sto(U)\ra\krn\vartheta_{_{\!\be S}}^{\e
*}&\ra&\displaystyle\bigoplus_{v\e\in\e B\e\setminus S}H^{1}(F_{\lbe
v}^{\e\rm{nr}},T)^{G_{k(v)}} \ra C_{\e T,\e F,\e S}\\\\
&\ra& C_{\e T,\e F,\e S}^{\e *}\ra \displaystyle\bigoplus_{v\e\in\e
B\e\setminus\e S}H^{1}\big(k(v),H^{1}(F_{\lbe
v}^{\le\rm{nr}},T)\big),
\end{array}
$$
where $\vartheta_{_{\!\be S}}^{\e *}$ is the map \eqref{theta*} and
$C_{\e T,\e F,\e S}^{\e *}$ is defined by \eqref{seq star}. In
particular, if $S\supset B$, then $C_{\e T,\e F,\e S}^{\e *}=C_{\e
T,\e F,\e S}$ and $\,\krn\vartheta_{_{\!\be S}}^{\e *}=\sto(U)$.
\end{lemma}
\begin{proof} This follows from the kernel-cokernel exact sequence
\cite{Mi}, Proposition I.0.24, p.19, associated to the pair of maps
\eqref{theta*}, using \eqref{7.5}.
\end{proof}

\begin{theorem} Let $0\ra T_{1}\ra Q\ra T\ra 0$ be a flasque
resolution of $T$. Assume that $S$ contains all primes of $F$ which
are wildly ramified in the minimal splitting field of $T_{1}$. Then
there exists a canonical exact sequence
$$\begin{array}{rcl}
0\ra\krn \vartheta_{\!S,\e 1}^{ *}\ra\sqo(U)\ra R\,\sto(U)&\ra&
C_{T_{1},\e F,\e S}^{*}\!\ra C_{Q,\e F,\e S}\ra C_{T,\e F,\e
S}^{R}\\\\
&\ra&\! \!\!\displaystyle\bigoplus_{v\in B_{1}\be\setminus S}
H^{1}(k(v),\Phi_{v}(T_{1})/{\rm{tors}})\ra 0,
\end{array}
$$
where $\vartheta_{\!S,\e 1}^{ *}$ and $C_{T_{1},\e F,\e S}^{*}$ are
given by \eqref{theta*} and  \eqref{seq star} for $T_{1}$,
respectively, $C_{T,\e F,\e S}^{R}$ is the $R$-equivalence class
group \eqref{R seq} and $B_{1}$ is the set of primes of $F$ where
$T_{1}$ has bad reduction.
\end{theorem}
\begin{proof} By the hypothesis and \cite{B}, Corollary 4.2.6, p.82,
$R^{1}j_{*}T_{1}=0$ for the smooth topology on $U$. Thus, by
\cite{B}, Theorem 5.3.1, p.99, and Lemma 4.8(c) applied to $Q$, for
each $v\notin S$ there exists an exact sequence
$$
\begin{array}{rcl}
0\ra
(\Phi_{v}(T_{1})\big(\e\overline{k(v)}\e\big)/\text{tors})^{G_{k(v)}}
&\ra&\Phi_{v}(Q)(k(v))\ra
\Phi_{v}(T)(k(v))\\\\
&\ra& H^{1}(k(v),\Phi_{v}(T_{1})/\text{tors})\ra 0.
\end{array}
$$
If $v\notin B_{1}$, then $\Phi_{v}(T_{1})/\text{tors}=X_{1}^{\vee}$
as $G_{k(v)}$-modules, where $X_{1}$ denotes the group of characters
of $T_{1}$ \cite{X}, Corollary 2.18. On the other hand, there exists
an injection $\text{Inf}\colon
H^{1}(k(v),X_{1}^{\vee})\hookrightarrow H^{1}(G_{v},X_{1}^{\vee})$,
and the latter group vanishes because $X_{1}$ is flasque \cite{CTS},
Lemma 1, p.179. We conclude that
$H^{1}(k(v),\Phi_{v}(T_{1})/\text{tors})=0$ for every $v\notin
B_{1}$. Thus there exists an exact sequence
\begin{equation}\label{7.8}
\begin{array}{rcl}
\displaystyle\bigoplus_{v\notin
S}(\Phi_{v}(T_{1})\big(\e\overline{k(v)}\e\big)/\text{tors})^{G_{k(v)}}
&\hookrightarrow&\displaystyle\bigoplus_{v\notin
S}\Phi_{v}(Q)(k(v))\ra \displaystyle\bigoplus_{v\notin
S}\Phi_{v}(T)(k(v))\\\\
&\ra&\displaystyle\bigoplus_{v\e\in\e B_{1}\setminus S}
H^{1}(k(v),\Phi_{v}(T_{1})/\text{tors})\ra 0.
\end{array}
\end{equation}
Consider now the exact commutative diagram
$$
\xymatrix{T_{1}(F)\ar[d]^(.4){\vartheta_{\!S,\e 1}^{
*}}\ar@{^{(}->}[r]
& Q(F)\ar@{->>}[r]\ar[d] & R\,T(F)\ar[d]^(.4){\vartheta_{_{\!\be S}}^{\lbe R}}\\
\displaystyle\bigoplus_{v\notin
S}(\Phi_{v}(T_{1})\big(\e\overline{k(v)}\e\big)/\text{tors})^{G_{k(v)}}
\ar@{^{(}->}[r]&\displaystyle\bigoplus_{v\notin
S}\Phi_{v}(Q)(k(v))\ar[r]&\displaystyle\bigoplus_{v\notin
S}\Phi_{v}(T)(k(v))}
$$
whose top row is \eqref{rat points}, its bottom row consists of the
first three terms of \eqref{7.8}, the left-hand vertical map is
\eqref{theta*} for the torus $T_{1}$ and the right-hand vertical map
is \eqref{theta R}. Applying the snake lemma to the above diagram
and using \eqref{7.8}, we immediately obtain the exact sequence of
the theorem.
\end{proof}

The following corollary of the theorem is a variant of Theorem 6.1.

\begin{corollary} Let the notations be as in the theorem. Assume that
$S\supset B_{1}$. Then there exists a canonical exact sequence
$$
0\ra\sto_{1}(U)\ra\sqo(U)\ra R\,\sto(U)\ra
C_{T_{1},\e F,\e S}\ra C_{Q,\e F,\e S}\ra C_{T,\e F,\e
S}^{R}\ra 0.
$$
\end{corollary}
\begin{proof} This follows from Theorem 7.4 using Lemma 7.3 for the torus $T_{1}$.
\end{proof}

\begin{remark} We have not attempted to extend Theorem 7.4 by
removing the hypothesis on $S$ there because we believe that the
result would be too complicated to be of much use (compare Theorem
5.3.1 on p.99 of \cite{B} with the significantly more complicated
Theorem 5.3.4 on p.101 of this reference). Perhaps a more sensible
problem to address for an arbitrary torus $T$ equipped with a
flasque resolution \eqref{f1} is the computation of the
corresponding Ono invariant $E_{\e f,\e S}(T)$. As is often the case
when computing numerical invariants, it seems reasonable to expect
that some of the complications will ``cancel out" in the course of
the computation. But we do not pursue this matter here.
\end{remark}

\section{Norm tori}

In this Section $T$ is any $F$-torus and $K/F$ is any finite Galois extension
such that $T_{K}$ is quasi-trivial. We keep the notations introduced in Section 2.

Define an $F$-torus $T^{\e\prime}$ by the exactness of the
sequence
\begin{equation}\label{normtorus}
0\ra T^{\e\prime}\ra R_{K/F}(T_{K})\overset{N}\longrightarrow T\ra
0,
\end{equation}
where $N$ is induced by the norm map $N_{K/F}\colon K^{*}\ra F^{*}$.
The torus $T^{\e\prime}$ is called the {\it norm} (or norm one)
torus determined by $T$ and $K/F$ and is often denoted by
$R^{(1)}_{K/F}(T_{K})$. See, e.g., \cite{B}, Theorem 0.4.4, p.16.
Clearly, \eqref{normtorus} induces an exact sequence
\begin{equation}\label{rat norm}
0\ra T^{\e\prime}(F)\ra T(K)\ra N_{K/F}\e T(K)\ra 0.
\end{equation}
Now, for each $v\notin S$, $N$ induces a map
$N_{v}\colon\prod_{w\mid v}\Phi_{w}(k(w))\ra\Phi_{v}(k(v))$ (see the
discussion after Theorem 6.1) and the following diagram commutes
$$
\xymatrix{T(K)\ar[d]^{N_{\lbe K/\lbe F}}\ar[r]^(.33){\vartheta_{_{\! S_{\be K}}}}&\displaystyle\bigoplus_{v\notin
S}\prod_{w\mid v}\Phi_{w}(k(w))\ar[d]^{\oplus\e N_{v}}\\
T(F)\ar[r]^(.36){\vartheta_{_{\be S}}}&\displaystyle\bigoplus_{v\notin
S}\Phi_{v}(k(v)).}
$$
Define
\begin{equation}\label{cl norm}
C_{\e T,\e F,\e S}^{N}=\cok\!\!\left[N_{K/F}\le T(K)\ra\displaystyle\bigoplus_{v\notin
S}\img N_{v}\right],
\end{equation}
where the map involved is induced by $\vartheta_{_{\be S}}$, and set
\begin{equation}\label{w}
W_{\e T,\e F,\e S}=\sto(U)\cap N_{K/F}\e T(K).
\end{equation}
By Proposition 4.1 for $T$, there exists a canonical exact sequence
\begin{equation}\label{n seq}
1\ra W_{\e T,\e F,\e S}\ra N_{K/F}\le T(K)\ra \displaystyle\bigoplus_{v\notin
S}\img N_{v}\ra C_{\e T,\e F,\e S}^{N}\ra 0.
\end{equation}
Further, there exist canonical maps
\begin{equation}\label{nu}
\nu\colon C_{\e T,\e K,\e S_{K}}\twoheadrightarrow C_{\e T,\e F,\e S}^{N}
\end{equation}
and
\begin{equation}\label{iota}
\iota\colon C_{\e T,\e F,\e S}^{N}\ra C_{\e T,\e F,\e S}
\end{equation}
which are induced by $\bigoplus_{\e v\notin S}\bigoplus_{\e w\mid v}\!\Phi_{w}(k(w))\overset{\oplus\e N_{v}}\longrightarrow
\bigoplus_{\e v\notin S}\img N_{v}$ and the inclusion $\bigoplus_{\e v\notin S}\img N_{v}\hookrightarrow\e\bigoplus_{\e v\notin S}\Phi_{v}(k(v))$, respectively.
Then
$$
\iota\circ\nu=N_{\e T,\e K/F,\e S}\colon C_{\e T,\e K,\e S_{K}}\ra C_{\e T,\e F,\e S}
$$
is norm map defined in Section 2 (this is a straightforward
verification). Since $\nu$ is surjective, we have
\begin{equation}\label{N}
\img \iota=N_{\e T,\e K/F,\e S}\, C_{\e T,\e K,\e S_{K}}.
\end{equation}

Now define
\begin{equation}\label{sha n}
\Sha_{N,\e S}(T\e)=\krn\!\!\left[\e\widehat{H}^{\e 0}(G,T(K))\ra\displaystyle\bigoplus_{v\notin
S}\widehat{H}^{\e 0}(G_{w_{v}},\Phi_{w_{v}}\be(k(w_{v})))\right],
\end{equation}
where the map involved is induced by $\vartheta_{_{\be S}}$.

\begin{remark} Using the fact that $C_{T_{K}\lbe ,\e w_{v}}=\widehat{H}^{-1}(G_{w_{v}},\Phi_{w_{v}}\be(k(w_{v})))=0$ (see the proof of Lemma 4.8), it is not difficult
to see that the above group is related to the more familiar one $\widehat{\!\be\Sha}^{0}_{S}(G,T(K))$ by an exact sequence
$$
0\ra\widehat{\!\be\Sha}^{0}_{S}(G,T(K))\ra \Sha_{N,\e S}(T\e)\ra\bigoplus_{v\notin S}\widehat{H}^{\e 0}(G_{w_{v}},\tto\lbe(\s O_{w_{v}}\be)).
$$
\end{remark}

\begin{lemma} There exists a canonical exact sequence
$$
1\ra\sto(U)/W_{\e T,\e F,\e S}\ra\Sha_{N,\e S}(T\e)\ra C_{\e T,\e F,\e S}^{N}\overset{\iota}\longrightarrow N_{\e T,\e K/F,\e S}\, C_{\e T,\e K,\e S_{K}}\ra 0,
$$
where $W_{\e T,\e F,\e S}$ is the group \eqref{w}, $C_{\e T,\e F,\e S}^{N}$ is defined by \eqref{cl norm} and $\iota$ is the map \eqref{iota}.
\end{lemma}
\begin{proof} This follows by applying the snake lemma to the diagram
$$
\xymatrix{N_{K/F}\e T(K)\ar[d]\ar@{^{(}->}[r]
& T(F)\ar@{->>}[r]\ar[d]^(.45){\vartheta_{_{\be S}}} & \widehat{H}^{\e 0}(G,T(K))\ar[d]\\
\displaystyle\bigoplus_{v\notin
S}\img N_{v}\ar@{^{(}->}[r]&\displaystyle\bigoplus_{v\notin
S}\Phi_{v}(k(v))\ar@{->>}[r]&\displaystyle\bigoplus_{v\notin
S}\widehat{H}^{\e 0}(G_{w_{v}},\Phi_{w_{v}}\be(k(w_{v})))}
$$
and using \eqref{n seq}, \eqref{N} and Proposition 4.1 for $T$.
\end{proof}

\bigskip

\begin{lemma} Assume that $S$ contains all primes of $F$ which are wildly ramified in the minimal splitting field of $T^{\e\prime}$.
Then there exists a canonical isomorphism
$$
\displaystyle\bigoplus_{v\notin S}\big(\Phi_{v}(T^{\e\prime}\e
)\big(\e\overline{k(v)}\e\big)/{\rm{tors}}\big)^{G_{k(v)}}=
\krn\!\!\left[\,\bigoplus_{\e v\notin S}\prod_{\e w\mid v}\Phi_{w}(k(w))\overset{\oplus\le N_{v}}\longrightarrow
\bigoplus_{\e v\notin S}\Phi_{v}(k(v))\right].
$$
\end{lemma}
\begin{proof} The hypothesis implies that $R^{1}j_{*}T^{\e\prime}=0$ for the smooth topology on $U$ \cite{B}, Corollary 4.2.6, p.82.
The lemma now follows by applying \cite{B}, Theorem 5.3.1, p.99, to
the exact sequence \eqref{normtorus} and using Lemma 4.8(c) for
$T_{K}$.
\end{proof}

\begin{theorem} Assume that $S$ contains all primes of $F$ which are wildly ramified in the minimal splitting field of $T^{\e\prime}$. Then there exists an exact sequence
$$
0\ra W_{\e T,\e F,\e S}/N_{K/F}\le\tto\be\big(\e\ut\e\big)\ra C_{\e T^{\e\prime}\be,\e F,\e S}^{\e *}\ra C_{\e T,\e K,\e S_{K}}\overset{\nu}\longrightarrow C_{\e T,\e F,\e S}^{N}\ra 0,
$$
where $W_{\e T,\e F,\e S}$ is the group \eqref{w}, $C_{\e
T^{\e\prime}\be,\e F,\e S}^{\e *}$ is given by \eqref{theta*} for
$T^{\e\prime}$ and $\nu$ is the map \eqref{nu}.
\end{theorem}
\begin{proof} This follows by applying the snake lemma to the exact commutative diagram
$$
\xymatrix{T^{\e\prime}(F)\ar[d]^(.45){\vartheta_{_{\!\be T^{\e\prime}\be,\e S}}^{*}}\ar@{^{(}->}[r]
& T(K)\ar[d]^(.45){\vartheta_{_{\!S_{\be K}}}}\ar@{->>}[r]^(.45){N_{K\lbe/\lbe F}} & N_{K/F}\e T(K)\ar[d]\\
\displaystyle\bigoplus_{v\notin
S}(\Phi_{v}(T^{\e\prime})\big(\e\overline{k(v)}\e\big)/\text{tors})^{G_{k(v)}}
\ar@{^{(}->}[r]&\displaystyle\bigoplus_{v\notin
S}\prod_{\e w\mid v}\Phi_{w}(k(w))\ar@{->>}[r]^(.55){\oplus N_{v}}&\displaystyle\bigoplus_{v\notin
S}\img N_{v},}
$$
where the top row is \eqref{rat norm}, the left-hand vertical map is \eqref{theta*} for the torus $T^{\e\prime}$ and the bottom row is exact by the previous lemma.
\end{proof}

\begin{remark} The above diagram and Proposition 4.1 show that
$\krn\vartheta_{_{\!\be T^{\e\prime}\be,\e S}}^{*}=\krn N_{\s O}$,
where $N_{\s O}\colon \tto\be\big(\e\ut\e\big)\ra\sto(U)$ is the
restriction of $N_{\be K/F}$ to $\tto\be\big(\e\ut\e\big)$. Thus, if
$S$ satisfies the hypothesis of the theorem, then Lemma 7.3 yields
an exact sequence
$$\begin{array}{rcl}
0\ra({\s T}^{\e\prime})^{\circ}(U)\ra\krn N_{\s O}&\ra&\displaystyle\bigoplus_{v\e\in\e
B^{\le\prime}\e\setminus S}H^{1}(F_{\lbe
v}^{\e\rm{nr}},T^{\e\prime}\e)^{G_{k(v)}} \ra C_{\e T^{\e\prime}\be,\e F,\e S}\\\\
&\ra& C_{\e T^{\e\prime}\be,\e F,\e S}^{\e *}\ra \displaystyle\bigoplus_{v\e\in\e
B^{\e\prime}\e\setminus\e S}H^{1}\big(k(v),H^{1}(F_{\lbe
v}^{\le\rm{nr}},T^{\e\prime}\e)\big),
\end{array}
$$
where ${\s T}^{\e\prime}$ is the N\'eron-Raynaud model of
$T^{\e\prime}$ over $U$ and $B^{\le\prime}$ denotes the set of
primes of $F$ where $T^{\e\prime}$ has bad reduction. The groups
$H^{1}(F_{\lbe
v}^{\e\rm{nr}},T^{\e\prime})^{G_{k(v)}}=\Phi_{v}(T^{\e\prime})(k(v))_{\text{tors}}$
were already considered in \cite{mrl}, \S4, when $T=\bg_{m,F}$ and
their orders were computed. See \cite{mrl}, Remark 4.4(b).
\end{remark}

The above remark and the theorem immediately yield

\begin{corollary} Assume that $S$ contains all primes of $F$ where $T^{\e\prime}$ has bad reduction. Then there exists an exact sequence
$$
0\ra W_{\e T,\e F,\e S}/N_{K/F}\le\tto\be\big(\e\ut\e\big)\ra C_{\e T^{\e\prime}\be,\e F,\e S}\ra C_{\e T,\e K,\e S_{K}}\overset{\nu}\longrightarrow C_{\e T,\e F,\e S}^{N}\ra 0,
$$
where $W_{\e T,\e F,\e S}$ is the group \eqref{w} and $\nu$ is the map \eqref{nu}.\qed
\end{corollary}

Assume now that $T$ is {\it coflasque}. Then \eqref{normtorus} is a coflasque resolution of $T^{\e\prime}$. Using this fact and Lemma 8.2, the above corollary yields the following formula for the Ono invariant $E_{\e c,\e S}(T^{\e\prime}\e)$.

\begin{corollary} Let $T$ be a coflasque $F$-torus, let $K/F$ be a finite Galois extension such that $T_{K}$ is quasi-trivial and let $T^{\e\prime}=R^{(1)}_{K/F}(T_{K})$ be the corresponding norm torus. Assume that $S$ contains all primes of $F$ where $T^{\e\prime}$ has bad reduction. Then
$$
E_{\e c,\e
S}(T^{\e\prime}\e)=\frac{[\!\be\Sha_{N,S}\le(T)]}{\big[\sto(U)\colon\!
N_{K/F}\tto\be\big(\ut\big)\big]\left[C_{\e T,\e F,\e S}\colon\!
N_{\e T,\e K,\e S_{\be K}}C_{\e T,\e K,\e S_{K}}\right]},
$$
where $\!\!\Sha_{N,S}(T)$ is the group \eqref{sha n}.\qed
\end{corollary}


\begin{thebibliography}{33}


\bibitem[1]{AW} Atiyah, M. and Wall, C.T.C. \emph{ Cohomology of
groups.} In: Algebraic Number Theory (J.W.S. Cassels and A.
Fr\"ohlich, Eds.), pp.94-115. Academic Press, London, 1967.



\bibitem[2]{BLR} Bosch, S., L\"{u}tkebohmert, W. and Raynaud, M.
\emph{ N\'eron Models.} Springer Verlag, Berlin 1989.


\bibitem[3]{BL} Bosch, S. and Liu, Q.\emph{ Rational points of the
group of components of a N\'eron model.} Manuscripta Math.
{\bf{98}} (1999), 275-293 .


\bibitem[4]{B} Brahm, B.\emph{ N\'eron-Modelle algebraischer Tori.}
Schriftenreihe des Mathematischen Instituts der Universit\"at
M\"unster. 3. Serie 31. M\"unster: Univ. M\"unster, Mathematisches
Institut; M\"unster: Univ. M\"unster, Fachbereich Mathematik und
Informatik (Dissertation). x, 134 pp. (2003).


\bibitem[5]{CTS} Colliot-Th\'el\`ene, J.-L. and Sansuc, J.-J.:
\emph{ $R$-equivalence sur les tores.} Ann. sci. \'Ec. Norm. Sup.,
$4^{\text{e}}$ s\'erie, vol. {\bf{10}} (1977), 175-229.



\bibitem[6]{C} Conrad, C.\emph{ Finiteness theorems for algebraic
groups over function fields.} Available at
http://math.stanford.edu/$\sim$conrad/papers/cosetfinite.pdf

\bibitem[7]{ELL} Edixhoven, S., Liu, Q. and Lorenzini, D.
\emph{ The $p$-part of the group of components of a N\'eron model.}
J. Algebraic Geom. {\bf{5}} (1996), 801-813.


\bibitem[8]{El} Elkik, R.\emph{ Solutions d'\'equations \`a coefficients
dans un anneau hens\'elien.} Ann. sci. \'Ec. Norm. Sup.,
$4^{\text{e}}$ s\'erie, vol. {\bf{6}} (1973), 553-603.


\bibitem[9]{EM} Endo, S. and Miyata, T.\emph{ On a classification of the function fields of algebraic tori.}
Nagoya Math. J. {\bf{56}} (1975), 85-104.


\bibitem[10]{CAmb} Gonz\'alez-Avil\'es, C.D.\emph{ Capitulation, ambiguous
classes and the cohomology of the units.} J. reine angew. Math. {\bf{613}} (2007), 75-97.


\bibitem[11]{mrl} Gonz\'alez-Avil\'es, C.D.\emph{ Chevalley's ambiguous class number formula for an
arbitrary torus.} Math. Res. Lett. {\bf{15}} (2008), no. 6, 1149-1165.



\bibitem[12]{leiden-talk} Gonz\'alez-Avil\'es, C.D.\emph{ N\'eron class groups and Tate-Shafarevich groups
of abelian varieties.} Notes for a talk delivered at Leiden University on 1/12/2009. Available at
www.math.leidenuniv.nl/$\sim$astolk/monday/notes/gonzales-neron.pdf


\bibitem[13]{I} Iwasawa, K.\emph{ On cohomology groups of units for
$\bz_{p}$-extensions } Amer. J. Math.  {\bf{185}} (1983), 189-200.



\bibitem[14]{K} Katayama, S.-I.\emph{ $E(K/k)$ and other arithmetical invariants for finite Galois extensions.}
Nagoya Math. J. {\bf{114}} (1989), 135-142.


\bibitem[15]{KM} Kunyavskii, B. and Moroz, B.\emph{ On integral models
of algebraic tori and affine toric varieties.} Available at
http://www.cs.biu.ac.il/$\sim\e$kunyav/publ.html

\bibitem[16]{Mi} Milne, J.S.:\emph{ Arithmetic Duality Theorems.}
Persp. in Math., vol. 1. Academic Press Inc., Orlando 1986.


\bibitem[17]{M} Morishita, M. \emph{ On $S$-class number relations
of algebraic tori in Galois extensions of global fields.} Nagoya
Math. J. {\bf{124}}, pp.133-144 (1991).


\bibitem[18]{NX} Nart, E. and Xarles, X.\emph{ Additive reduction of
algebraic tori.} Arch. Math. {\bf{57}} (1991), 460-466.


\bibitem[19]{Nis} Nisnevich, Ye.
\emph{ Etale Cohomology and Arithmetic of Semisimple Groups.}
Thesis, Harvard University, 1982.


\bibitem[20]{Nis2} Nisnevich, Ye.\emph{ The completely decomposed topology on schemes and associated
descent spectral sequences in algebraic $K$-theory.} Algebraic
$K$-theory: connections with geometry and topology (Lake Louise, AB,
1987). NATO Adv. Sci. Inst. Ser. C Math. Phys. Sci., 279, pp.
241--342. Kluwer Acad. Publ., Dordrecht, 1989.


\bibitem[21]{Oe} Oesterl\'e, J.:\emph{ Nombres de Tamagawa et groupes unipotentes
en caract\'erisque $p$.} Invent. Math. {\bf{78}} (1984), 13-88.


\bibitem[22]{O} Ono, T.: \emph{ On some class number relations for Galois extensions} Nagoya
Math. J. {\bf{107}} (1987), 121-133.


\bibitem[23]{PR} Platonov, V. and Rapinchuk, A.:
\emph{ Algebraic Groups and Number Theory.} Academic Press Inc., Boston 1994.


\bibitem[24]{R} Rose, J.S.\emph{ A Course on Group Theory.}
Cambridge Univ. Press, Cambridge 1978.



\bibitem[25]{Sa} Sasaki, R.:\emph{ Some remarks to Ono's theorem on a generalization of Gauss'
genus theory.} Nagoya Math. J., {\bf{111}} (1988), 131-142.


\bibitem[26]{San} Sansuc, J.-J.:\emph{ Groupe de Brauer et
aritm\'etique de groupes alg\'ebriques lin\'eares sur un corps de
nombres.} J. reine angew. Math. {\bf{327}} (1981), 12-80.


\bibitem[27]{S} Serre, J.-P.\emph{ Local Fields.}
Grad. Texts in Math. {\bf{67}}, Springer-Verlag, 1979.


\bibitem[28]{Sh1} Shyr, J.-M.\emph{ A generalization of Dirichlet's Unit Theorem.}
J. Number Theory {\bf{9}} (1991), 213-217.


\bibitem[29]{Sh2} Shyr, J.-M.\emph{ On some class number relations of algebraic tori.}
Mich. Math. J. {\bf{24}} (1977), 365-377.


\bibitem[30]{T} Tamme, G. \emph{ Introduction to \'Etale
Cohomology.} Springer-Verlag, Berlin, 1994.

\bibitem[31]{V} Voskresenskii, V. \emph{ Algebraic Groups and Their
Birational Invariants.} Trans. Math. Monogr. vol. {\bf{179}}, Amer.
Math. Soc. 1998.

\bibitem[32]{X} Xarles, X.\emph{ The scheme of connected components
of the N\'eron model of an algebraic torus.} J. reine angew. Math.
{\bf{437}} (1993), 167-179.


\bibitem[33]{Z} Zassenhaus, H.\emph{ The Theory of Groups} (2nd Ed).
Chelsea Pub. Co., New York 1958.


\end{thebibliography}
\end{document}